\documentclass[twoside,a4paper,11pt]{article}
\usepackage[]{setspace}
\usepackage{amsmath, amssymb, amsthm, sectsty}
\usepackage{graphicx}%
\usepackage{bigints}
\usepackage[in]{fullpage}
\usepackage{chngcntr}
\usepackage{mathtools}
\usepackage{chngcntr}
\usepackage{cite}	
\usepackage{hyperref,xcolor}
\usepackage[OT1]{fontenc}
\usepackage{titlesec,titletoc}
\usepackage{tabu,longtable}
\usepackage{fancyhdr,verbatim}

\numberwithin{equation}{section}

\definecolor{dark-green}{rgb}{0.1,1,0.3}
\hypersetup{
colorlinks=true,
    pdfborder={0 0 0},
citecolor=dark-green,
linkcolor=blue
}

\newcommand{\ab}[1]{\left\lvert #1 \right\rvert}
\newcommand{\md}[1]{(\text{mod}\ #1)}
\newcommand{\sg}{\sigma}
\newcommand{\al}{\alpha}
\newcommand{\eps}{\epsilon}
\newcommand{\dl}{\delta}
\newcommand{\kp}{\kappa}
\newcommand{\lb}{\lambda}

\newcommand{\mb}{\mathbb}
\newcommand{\bt}{\beta}
\newcommand{\gm}{\gamma}
\newcommand{\vp}{\varphi}
\newcommand{\ds}{\displaystyle}
\newcommand{\lm}{\limits}
\newcommand{\st}{\substack}
\newcommand{\cg}{\equiv}
\newcommand{\nd}{\noindent}

\newcommand{\ov}{\overline}
\newcommand{\sr}{\sqrt}

\newcommand{\fr}{\frac}

\newcommand{\Og}{\Omega}
\newcommand{\sd}{\sideset}
\newcommand{\tps}{\texorpdfstring}

\newtheorem{thm}{Theorem}[section]
\newtheorem{cor}[thm]{Corollary}
\newtheorem{lem}[thm]{Lemma}
\newtheorem{prop}[thm]{Proposition}

\theoremstyle{definition}
\newtheorem{defn}[thm]{Definition}
\newtheorem{ntn}[thm]{Notation}
\newtheorem{rem}[thm]{Remark}

\fancypagestyle{plain}{%
\fancyhead{} 
}

\AtEndDocument{\bigskip{\normalsize%
  \textsc{The Institute of Mathematical Sciences, Chennai} \par
	\nd
  \textit{E-mail address}, R. Balasubramanian: \texttt{balu@imsc.res.in} \par
	\vskip 0.1in
	\nd
	\textsc{The Institute of Mathematical Sciences, Chennai} \par  
	\nd
  \textit{E-mail address}, Priyamvad Srivastav: \texttt{priyamvads@imsc.res.in} \par
}}

\allsectionsfont{\centering}
\begin{document}
\title{On Selberg's approximation to the twin prime problem}
\author{R. Balasubramanian and Priyamvad Srivastav}
\date{}
\maketitle{}
\bigskip
\begin{abstract}
In his Classical approximation to the Twin prime problem, Selberg proved that for $x$ sufficiently large, there is an $n \in (x,2x)$ such that $2^{\Og(n)}+2^{\Og(n+2)} \leq \lb$ with $\lb=14$, where $\Og(n)$ is the number of prime factors of $n$ counted with multiplicity. This enabled him to show that for infinitely many $n$, $n(n+2)$ has atmost $5$ prime factors, with one having atmost $2$ and the other having atmost $3$ prime factors. By adopting Selberg's approach and using a refinement suggested by Selberg, we improve this value of $\lb$ to about $\lb=12.59$.
\end{abstract}
\bigskip
\section{Introduction}
The Twin prime conjecture is one of the oldest unsolved problems in Number theory. A statement very simple to understand has eluded numerous attempts by the most adept of mathematicians. In fact, we seem nowhere close to settling this problem. 
The best known approximation to this problem is due to J.R. Chen\cite{JRC} in 1973, which states that there are infinitely many primes $p$, for which $p+2$ has atmost two prime factors. \vskip 0.01in
In the last few years, there have been major developments on a related problem, called the Bounded gaps problem. This problem asks whether the quantity $H_m=\liminf\lm_{n \to \infty} (p_{n+m}-p_n)$ is finite and provide 
an upper bound for the same. In 2006, Goldston, Pintz and Yildirim(GPY) in \cite{GPY} proved
$$ \liminf\lm_{n \to \infty} \fr{p_{n+1}-p_n}{\log p_n}=0$$
thus settling a long standing conjecture. More interesting than their result was the method they used to arrive at this result. This method, known as the GPY sieve method, 
caught the attention of many experts in the field. In the same year 2006, Y. Motohashi and J. Pintz \cite{GPYs}, introduced a smoothed version of the GPY sieve. Their method just fell short of proving the bounded 
gaps problem. It was finally in 2013 that Zhang \cite{YZ} showed $H_1 \leq 7 \times 10^7$, marking an important breakthrough in the subject. This bound was subsequently lowered to about $4680$ by the Polymath 8a \cite{PM8a} project. 
In his paper, Zhang used the GPY sieve, alongwith a modified version of the Bombieri-Vinogradov Theorem. It was the modified version of the Bombieri Vinogradov Theorem, which allowed him to deduce his results.\\  \vskip 0.03in
Recently, in his preprint titled `small gaps between primes', James Maynard \cite{JM} introduced a refinement to the $GPY$ method by using multidimensional sieve weights. 
Under this refinement, Maynard shows that $H_1 \leq 600$ and that $H_m \ll m^3e^{4m}$. This result was also proved independently by Tao. In his paper, Maynard has only used the 
Bombieri-Vinogradov Theorem and the paper does not incorporate any of the technology involved in the Zhang's paper. The Polymath 8b \cite{PM8b} project, which extends Maynard's methods, has successfully brought down $H_1$ to about $H_1 \leq 246$. 
Under the Generalised Elliott-Halbertsam (GEH) conjecture, Polymath 8b shows that $H_1 \leq 6$. \\ \vskip 0.01in 
The use of multidimensional sieve weights actually dates back to Atle Selberg, who suggested their use in his approximations to the Twin prime problem \cite[page 240]{AS}. 
In this unpublished manuscript, Selberg considers the sum
\begin{equation}
\label{1dsum}
\sum\lm_{\st{x<n \leq 2x\\n \cg -1 \md{6}}} \left( 1-\fr{2^{\Og(n)}+2^{\Og(n+2)}}{\lb} \right)
\left( \sum\lm_{\st{d \mid n(n+2)\\ d \leq z}}\lb_d \right)^2
\end{equation}
where $z=x^{1/3-\eps}$. He succeeded in showing that (\ref{1dsum}) is positive for any $\lb>14$. When this sum is positive, it means that there is an $x<n \leq 2x$ such that $2^{\Og(n)}+2^{\Og(n+2)} \leq 14$, which implies that $n(n+2) \in P_5$. 
Since this happens for all $x$, there are infinitely many $n$ for which $n(n+2) \in P_5$. \vskip 0.01in
Note that if $\lb$ can somehow be brought below $12$, one could show the existence of infinitely many $n$ for which $n(n+2) \in P_4$, since $12=2^2+2^3$. \vskip 0.07in
By adjusting the sieve weights suitably, Gerd Hofmeister(unpublished) was able to bring down $\lb$ to ``about $13$'', though Selberg is not clear about what exactly ``about $13$'' means.\vskip 0.01in
Further, in order to improve the value of $\lb$, Selberg \cite[page 245]{AS} suggested two-dimensional sieve weights of the form
\begin{equation} \left( \sum\lm_{\st{d_1 \mid n \\ d_2 \mid n+2}}\lb_{d_1,d_2} \right)^2  \label{2dweights} \end{equation}
\vskip 0.2in
In this paper, we adopt Selberg's approach and improve the value of $\lb$ by using the weights as given in (\ref{2dweights}). However, this approach falls somewhat short of bringing $\lb$ below $12$. \vskip 0.07in 
We consider the sum
\begin{equation} \label{msum}
\sum\lm_{\st{n \sim x\\n \cg v_0 \md{W}}}\left( 1-\fr{2^{\Og(n)}+2^{\Og(n+h)}}{\lb} \right) \left( \sum\lm_{\st{d_1 \mid n\\d_2 \mid n+h}}\lb_{d_1,d_2} \right)^2=S_1 - S_2 /\lb
\end{equation}
where $h$ is an even number and $v_0$ is chosen such that $(n(n+h),W)=1$. Following Maynard\cite{JM}, we have defined 
\begin{equation} D=\log\log\log{x} \quad \text{and} \quad W=\prod\lm_{p \leq D}p \label{DW} \end{equation}
The sieve weights $\lb_{d_1,d_2}$ are supported on
\begin{equation}
\label{support}
S(z)=\{(d_1,d_2): \mu^2(d_1d_2W)=1,  \max\{d_1^{2/3} d_2 , d_1d_2^{2/3}\} \leq z\} 
\end{equation}

In this paper, we show that
\begin{thm}
\label{mainthm}
(\ref{msum}) is positive for any $\lb>12.59$. 
\end{thm}
\vskip 0.1in
The paper consists of three major sections. In the first section, we introduce the Perron's formula and use it to provide asymptotic estimates for partial sums of a certain class of arithmetic functions. \vskip 0.07in

In the second section, we deal with partial divisor sums like $\tau(n)$ and $2^{\Og(n)}$ in arithmetic progressions. This is because the quantity $S_2$ in (\ref{msum}) turns into a combination of sums of the function $2^{\Og(n)}$ in arithmetic progressions. 
We write expressions for such divisor sums in arithmetic progressions in terms of suitable Kloostermann and Exponential sums. These results are essentially due to Selberg in his unpublished manuscript \cite[page 241]{AS}. 
The Weil's bound for exponential sums allows us to obtain asymptotic formulae for the sum
$$ \sum\lm_{\st{n \leq x\\n \cg a \md{m}}}2^{\Og(n)} $$
with a good error term. \vskip 0.07in

In the final section, we obtain asymptotic expressions for $S_1$ and $S_2$ (See \ref{S1} and \ref{S2}) to compute a suitable value for $\lb$. 
In the computations, Corollary \ref{sumS} from Section 1 is frequently invoked as it provides asymptotic formula for just the sums we encounter. 
Moreover, these calculations and estimates follow many ideas from Maynard's preprint \cite{JM}. 
In order to optimise the choice of sieve weights, we rely on a sage program. (See \ref{Pchoice}).
\vskip 0.75in
\section*{Notation}
\thispagestyle{empty}
Throughout this paper, $h$ remains a fixed even number. The notation $n \sim x$ means that $n<x \leq 2x$. The symbol $p$ is reserved for a prime number. The numbers $\eps$, $x$ and $z$ will always be positive real numbers with $z \leq x$. Many times, we shall assume that $x$ is sufficiently large and that $\eps$ is sufficiently small. We write $(a,b)$ and $[a,b]$ for the GCD and LCM of positive integers $a$, $b$ respectively. In many places, particularly Section 3, we write $f(a,b)$(or $f[a,b]$) to denote $f((a,b))$(or $f([a,b])$ ) to simplify notation. For arithmetic functions $f$ and $g$, $f*g$ denotes the dirichlet convolution of $f$ and $g$. The symbols $'O'$, $'\ll'$ denote the usual big oh notation and $'o'$ denotes the little oh notation. In many places, particularly Section 3, the $O$-constants depend on $\eps$ or the functions $P$, $Q_1$ or $Q_2$, which will all be bounded functions. For the sake of simplicity, we do not specify the dependence of these $O$-constants.\\ \vskip 0.03in
The following table has a list of arithmetic functions. All of these functions barring $\Og(n)$ are multiplicative. In some cases, we shall define tthe function on primes as we are only concerned with their value at squarefree integers. The functions $f$, $f_1$, $g_1$, $h$, $h_1$ and $h_2$ appear only in Section 3. \vskip 0.4in
\tabcolsep 1cm
\thispagestyle{empty}
\begin{longtable}{l l}
{\bf \normalsize Function} & {\bf \normalsize Description} \\[3ex]
$\mu$ & The Moebius function \\[1ex]
$id$ & Defined by $id(n)=n$ \\[1ex]
$\vp$ & Euler Totient function\\[1ex]
$\tau$ & The Divisor function \\[1ex]
$\Og$ & The additive function given by $\Og(p^{k})=k$ \\[1ex]
$g$ & Defined by $g(p)=\fr{p(p-1)}{p-2}$\\[1ex]
$f$ & Defined by $f(p)=p/2$\\[1ex]
$f_1,g_1$ & Satisfy $f=f_1*1 \ $ and $\ g=g_1*1$\\[1ex]
$h$ & $h(p)=1-3/(p+2)$ \\[1ex]
$h_1,h_2$ & $h_1(p)=1-3/p+2/p^2 \ $  and  $ \ h_2(p)=1-2/p+2/p^2$\\[1ex]
\end{longtable}

\pagestyle{fancy}
\newpage
\section{Preliminary Results}
In this section, we review some of the well known results for partial sums of arithmetic functions. 
We define a class of arithmetic functions and give asymptotic estimates for their partial sums. We state a well known result

\begin{lem}
\label{logp}
\begin{equation*}
\sum\lm_{p \leq x} \fr{\log{p}}{p}=\log{x}+O(1)
\end{equation*}
\end{lem}
\begin{proof}
 For a proof, see \cite[Pg 50]{MV}
\end{proof}

\medskip
This leads to an immediate Corollary.
\begin{cor}
\label{logp/logz}
Let $P$ be a continuously diferentiable function on $[0,\rho]$ and let $x^{1/\rho} \leq z \leq x$. Then we have
$$ \sum\lm_{p \leq x}\fr{\log{p}}{p}P\left( \fr{\log{p}}{\log{z}} \right)=\left( \log{z}+O(1) \right)
\int\lm_0^{\fr{\log{x}}{\log{z}}} P(t) \, dt $$
\end{cor}
\begin{proof}
We apply partial summation to the function $\fr{\log{n}}{n}1_{\mb{P}}(n) P \left( \fr{\log{n}}{\log{z}} \right)$. Since we know from Lemma \ref{logp} that 
$\sum\lm_{p \leq x}\fr{\log{p}}{p}=\log{x}+O(1)$, we get
\begin{equation*}
\begin{split}
\sum\lm_{p \leq x}\fr{\log{p}}{p} P\left( \fr{\log{p}}{\log{z}} \right)
&=\int\lm_1^x P\left(\fr{\log{t}}{\log{z}}\right)d(\log{t}+O(1))\\
&=\int\lm_1^x P\left( \fr{\log{t}}{\log{z}}\right)  \, \fr{dt}{t}  + O(1) + 
O\left( \int_1^x P'\left( \fr{\log{t}}{\log{z}} \right) \, d\left(\fr{\log{t}}{\log{z}}\right) \right)
\end{split}
\end{equation*}
Note that the main term above is clearly $ \ds (\log{z}) \int_0^{\fr{\log{x}}{\log{z}}} P(t) \, dt $
and the error term is $O(1)$. This completes the proof.
\end{proof}
\medskip
We now state the following Lemma.
\begin{lem}
\label{contour}
Let $c>0$ and $x>0$ be a real number which is not an integer. Then for any $T>0$, we have
$$ \left| \fr{1}{2 \pi i}\int_{c-iT}^{c+iT} \fr{x^s}{s} \, ds - \delta(x) \right| \leq x^c \min{\left\{1,\fr{1}{T |\log{x}|} \right\}}$$ 
where $$\delta(x)=\begin{cases} 0 & \text{if} \ 0<x<1 \\ 1 & \text{if} \ x>1 \end{cases}$$
\end{lem}
\begin{proof}
For a proof, see \cite[Proposition 5.54, Pg 151]{IK} 
\end{proof}
This leads to an immediate Corollary, namely the Perron's formula.
\begin{prop}[Perron's formula]
 Let $F(s)$ be the Dirichlet series of $f$ which converges absolutely for $\sg>\sg_a \geq 0$. Let $c>\sg_a$and let $x>0$ be a non-integer. Then
 \begin{equation*}
 \sum\lm_{n \leq x} f(n)=\fr{1}{2 \pi i} \int_{c-iT}^{c+iT} F(s)\fr{x^s}{s} \, ds + O\left(x^c\sum\lm_{n=1}^{\infty}
\fr{|f(n)|}{n^c} \min \left\{ 1,\fr{1}{T|\log{\fr{x}{n}}|} \right\} \right) 
 \end{equation*}
\end{prop}
\begin{proof}
We have
\begin{equation*}
\fr{1}{2 \pi i}\int_{c-iT}^{c+iT} F(s)\fr{x^s}{s} \, ds=\fr{1}{2 \pi i}\int_{c-iT}^{c+iT} \left( \sum\lm_{n=1}^{\infty} \fr{f(n)}{n^s}\right) \fr{x^s}{s} \, ds
\end{equation*}
Since the contour formed by the line joining $c-iT$ and $c+iT$ is compact, it follows that $F(s)$ converges uniformly here. We can therefore swap the order of integration and 
summation to obtain
\begin{equation*}
\fr{1}{2 \pi i}\int_{c-iT}^{c+iT} F(s) \fr{x^s}{s} \, ds=\fr{1}{2 \pi i}\sum\lm_{n=1}^{\infty} f(n) \int_{c-iT}^{c+iT} \fr{1}{s}\left(\fr{x}{n}\right)^s \, ds 
\end{equation*}
By Lemma \ref{contour}, we have
\begin{equation*}
\fr{1}{2 \pi i}\int_{c-iT}^{c+iT}\fr{1}{s}\left(\fr{x}{n}\right)^s \, ds=\begin{cases} 
1+O\left( \fr{x^c}{n^c} \min\left\{ 1, \fr{1}{T|\log{\fr{x}{n}}|} \right\} \right) & n \leq x \\ 
O\left( \fr{x^c}{n^c} \min\left\{ 1, \fr{1}{T|\log{\fr{x}{n}}|} \right\} \right) & n>x \end{cases}
\end{equation*}
From the above two relations, the result follows.
\end{proof}
\medskip

We define a class of multiplicative arithmetic functions and prove asymptotic formula for partial sums of the same. 
These are the type of functions we shall be encountering in the main computations of Section 3. 
\begin{defn}
\label{omegak}
For an integer $k \geq 1$, we define $\Og_k$ to be the set of all multiplicative arithmetic functions $f$ which are supported on the squarefree integers and the Dirichlet series $F(s)$ of $f$ is of the form 
$F(s)=\zeta^k(1+s)G(s)$, where $G(s)$ is given by an absolutely convergent series in $\sg \geq -\dl$, for some $\dl>0$.
\end{defn}

\medskip
\begin{prop}
\label{residue}
Let $f \in \Og_k$ be an arithmetic function with Dirichlet series $F(s)=\zeta^k(s+1)G(s)$. Then for any non-integer $x > 1$,
\begin{equation*}
\sum\lm_{n \leq x}f(n)=G(0)\fr{(\log{x})^k}{k!}+O\left( \log^{k-1} x \right)
\end{equation*}
Note that here $G(s)=\prod\lm_{p}\left( 1-\fr{1}{p^{1+s}} \right)^k \left( 1+\fr{f(p)}{p^s} \right)$.
\end{prop}
\begin{proof}
From (13.10) of \cite[Pg 353]{IV}, we have for any $k \geq 1$,
\begin{equation}
\label{divsum}
\sum\lm_{n \leq x} \tau_k(n) = x P_k(\log x) + O_{\eps}(x^{\theta_k+\eps})
\end{equation}
where $0<\theta_k<1$, $\tau_k=\underbrace{1*1* \dots *1}_{k \text{ times}}$ and $P_k(\log x)=Res_{s=1} \zeta^k(s) \fr{x^{s-1}}{s}$ is a polynomial of degree $k-1$ in $\log x$, 
with leading term $\fr{(\log x)^{k-1}}{(k-1)!}$. By partial summation, we obtain
\begin{equation}
\label{tauk}
\sum\lm_{n \leq x} \fr{\tau_k(n)}{n}=\left( 1 + O\left( \fr{1}{\log x} \right) \right)\fr{(\log x)^k}{k!}
\end{equation}
Since the Dirichlet series of $\tau_k/id$ is $\zeta^k(s+1)$, we have $f=\tau_k/ id * g$. Therefore, by the convolution method,
\begin{equation}
\begin{split}
\label{fsum}
\sum\lm_{n \leq x} f(n)&=\sum\lm_{ab \leq x} g(a) \fr{\tau_k(b)}{b}=\sum\lm_{a \leq x}g(a) \sum\lm_{b \leq x/a} \fr{\tau_k(b)}{b}
=\sum\lm_{a \leq x} g(a) \left( 1 + O\left( \fr{1}{\log x/a} \right) \right) \fr{(\log \fr{x}{a})^k}{k!}\\
&= \fr{(\log x)^k}{k!} \sum\lm_{a \leq x} g(a) \left( 1-\fr{\log a}{\log x} \right)^k+O\left( \log^{k-1} x \right)
\end{split} 
\end{equation}
Since $G(s)$ is well defined at $s=0$, it follows that $\sum\lm_{a \leq t}g(a)=G(0)+A^*(t)$, where $A^*(x)=\sum\lm_{n>x}g(n)$. Since $G(s)=\sum\lm_n \fr{g(n)}{n^s}$ converges absolutely 
for $\sg \geq -\dl$, it follows that $A^*(t) \ll t^{-\dl}$. Letting $Q(x)=(1-x)^k$, we have
\begin{equation*}
\begin{split}
\sum\lm_{n \leq x} g(n) Q\left( \fr{\log n}{\log x} \right)&=\int_1^x Q\left( \fr{\log t}{\log x} \right) d(G(0)-A^*(t))\\
&=A Q(0) +O(x^{-\dl})+ 
O\left(\int_1^x \fr{1}{t^{1+\dl} \log x} Q'\left( \fr{\log t}{\log x} \right) \, dt \right)\\
&=G(0) Q(0) + O(x^{-\dl})
\end{split}
\end{equation*}
Substituting the above expression into (\ref{fsum}), we obtain the desired result.
\end{proof}
\begin{rem}
Actually, the asymptotic formula for $\sum\lm_{n \leq x} f(n)$ in the previous Proposition is of the form $C Q(\log x) + 
O(x^{-\theta})$, where $C$ is the appropriate constant, $Q$ 
is a polynomial and $0<\theta <1$. Since the expression given in Proposition \ref{residue} suffices for our purposes, we avoid writing the main term as a polynomial and instead write with an error 
$O\left( \fr{1}{\log x} \right)$.
\end{rem}

\medskip
We now redefine the constant $G(0)$ occuring in the above Proposition.
\begin{defn}
Let $f \in \Og_k$ and let $\ov{f}=f*1$. Then for any positive integer $m$, we define
\begin{equation}\label{cmf} c(m,f)=\prod\lm_{p \nmid m} \left(1-\fr{1}{p} \right)^k \ov{f}(p)  \end{equation}
\end{defn}
\medskip
Proposition \ref{residue} leads to the following Corollary, which gives us the asymptotic formula for partial sums of functions in $\Og_k$.
\begin{prop}
\label{asympest}
Let $f \in \Og_k$ and let $\ov{f}=f*1$. Then for any positive integer $m$ and $(d,m)=1$, we have
\begin{equation*}
\sum\lm_{\st{n \leq z\\n \cg 0 \md{d} \\ (n,m)=1}} f(n) =  \fr{f(d)}{\ov{f}(d)}
\left(\fr{\vp(m) \log z}{m}\right)^k \ \fr{c(m,f)}{k!}\left( 1-\fr{\log d}{\log z} \right)^k + O(\log^{k-1}z)
\end{equation*}
\end{prop}
\begin{proof}
Firstly, we note that
\begin{equation*}
\sum\lm_{\st{n \leq z\\ n \cg 0 \md{d} \\ (n,m)=1}} f(n)=f(d) \sum\lm_{\st{n \leq z/d \\ (n,dm)=1}} f(n) 
\end{equation*}
We now apply Proposition \ref{residue} to the function $f(n) 1_{(n,dm)=1}$ with $x=z/d$. Note that the Dirichlet series of this function is $\zeta^k(1+s) G_{dm}(s)$, 
where $G_{dm}(s)=G(s)\prod\lm_{p \mid dm}\left(1+\fr{f(p)}{p^s}\right)^{-1}$. We therefore obtain,
\begin{equation*}
 \sum\lm_{\st{n \leq z\\ n \cg 0 \md{d} \\ (n,m)=1}} f(n)=f(d)\fr{\log^k z/d}{k!} G_{dm}(0) + O\left( \log^{k-1} z/d \right)
\end{equation*}
Writing
$$ G_{dm}(0)=G(0)\prod\lm_{p \mid dm} (1+f(p))^{-1}=\fr{1}{\ov{f}(d)}\fr{\vp^k(m)}{m^k}\prod\lm_{p \nmid m} (1-1/p)^k (1+f(p)) $$
we obtain the desired result.
\end{proof}

\medskip
The next Theorem gives us the asymptotic formula for functions in $\Og_k$ accompanied by a smoothing function.
\begin{thm}
\label{estcong}
Let $f \in \Og_k$ and let $\ov{f}=f*1$. Suppose $0<w_1< w_2 \ll z^{\rho}$ for some $\rho>0$. Let $d \leq w_2$ and let $P$ be a continuously differentiable function on $[0,\rho]$. Then for any positive integer $m$,
\begin{equation*}
\begin{split}
\sum\lm_{\st{w_1<n \leq w_2 \\ n \cg 0 \md{d}\\ (n,m)=1}} f(n)P\left( \fr{\log{n}}{\log{z}} \right)
&=\left(\fr{\vp(m)\log z}{m}\right)^k \fr{f(d)}{\ov{f}(d)}  \fr{c(m,f)}{(k-1)!}
\int\lm_{\max\left\{\fr{\log{w_1}}{\log{z}},\fr{\log{d}}{\log{z}}\right\}}^{\fr{\log{w_2}}{\log{z}}} P(t) \left( t-\fr{\log{d}}{\log{z}} \right)^{k-1} \, dt \\
&\quad + O(\log^{k-1} z)
\end{split}
\end{equation*}
\end{thm}
\begin{proof}
Let $f_{m,d}$ be defined as in the previous Proposition \ref{asympest}. The required sum to be estimated then is
$$ \sum\lm_{n \leq z} f_{m,d}(n) P\left( \fr{\log{n}}{\log{z}} \right) $$
We have seen in Proposition \ref{asympest} that for any $x>0$
$$ \sum\lm_{n \leq x}f_{m,d}(n)=M(x)+E(x)$$
where $M(x)$ is the main term and $E(x)$ is the error term. We also know that $E(x) \ll \log^{k-1}x$ and that $M(x)$ is a differentiable function on $\mb{R}^{>0}$. 
We apply partial summation to get
\begin{equation*}
\sum\lm_{w_1<n \leq w_2}f_{m,d}(n)P\left( \fr{\log{n}}{\log{z}} \right)=\int\lm_{w_1}^{w_2}P\left( \fr{\log{t}}{\log{z}}\right)d(M(t)+E(t)) \, dt
\end{equation*}
The main term above is
$$ \int\lm_{w_1}^{w_2} M'(t)P\left( \fr{\log{t}}{\log{z}} \right) \, dt $$
and the error term is
\begin{equation*}
\begin{split}
& O\left(|E(w_2)|+|E(w_1)| + \int\lm_{w_1}^{w_2} \fr{|E(t)|}{t \log{z}} \, dt \right)
\end{split}
\end{equation*}
Since $ E(t) \ll \log^{k-1} t$, it follows that the error term above is $O\left( \log^{k-1}z\right)$. Substituting the value 
of the main term $M(t)$ from Proposition \ref{asympest}, the main term above is
\begin{equation*}
\begin{split}
& \quad \fr{1}{k!}\fr{f(d)}{\ov{f}(d)}\left(\fr{\vp(m)\log{z}}{m}\right)^k c_m(f)\int\lm_{\max\{w_1,d\}}^{w_2} \left( \fr{\log{t}}{\log{z}} - \fr{\log{d}}{\log{z}} \right)^{k-1} P\left(\fr{\log{t}}{\log{z}} \right) \, \fr{k \ dt}{t \log{z}}
\end{split}
\end{equation*}
The change of variable $ \ds t \to \fr{\log{t}}{\log{z}}$ yields us the desired main term. 
\end{proof}

\begin{defn}
For any real number $s$, we define
\begin{equation}
\label{etadef}
\eta(s)=\min\left\{ 1-\fr{2s}{3},\fr{3(1-s)}{2} \right\}
\end{equation}
One easily sees that
\label{3/5}
\begin{equation*}
\eta(s)=\begin{dcases} 1-\fr{2s}{3} & \text{if} \ s \leq \fr{3}{5} \\ \fr{3(1-s)}{2} & \text{if} \ s \geq \fr{3}{5} \end{dcases}
\end{equation*}
\end{defn}
\medskip
We state the next Lemma without proof.
\medskip
\begin{lem}
\label{etacn}
Let $d_1$, $d_2$ be positive integers. Let $z \leq x$, $S(z)$ be as given in (\ref{support}) and $W$ be as given in (\ref{DW}). Then
\begin{equation*}
\begin{split}
(d_1,d_2) \in S(z) &\iff \mu^2(d_1d_2W)=1 \ \text{and} \ \max\{d_1d_2^{2/3},d_1^{2/3}d_2\} \leq z \\
&\iff \mu^2(d_1d_2W)=1 \ \text{and} \ d_1 \leq z^{\eta\left( \fr{\log{d_2}}{\log{z}} \right)}
\end{split}
\end{equation*}
\end{lem}
\begin{defn}
Let $s_1, s_2 \in [0,1]$. We define the region
\begin{equation}
\label{Ts1s2}
T_{s_1,s_2}=\left\{(x_1,x_2) \in \mb{R}^2: s_i \leq x_i, \ x_1+\fr{2x_2}{3} \leq 1, \ \fr{2x_1}{3}+x_2 \leq 1\right\} 
\end{equation}
We shall denote the region $T_{0,0}$ by $T$.
\end{defn}
\medskip
The next Corollary is the key result of this section. In our computations later on, we shall be invoking this lemma quite frequently.
\begin{cor}
\label{sumS}
Let $f_1 \in \Og_{k_1}$ and $f_2 \in \Og_{k_2}$ for positive integers $k_1$ and $k_2$. Let $f=f_1*f_2$ and let $P:T \to \mb{R}$ be a function differentiable in each variable. Let $d_1, d_2$ be positive integers
and let $s_i=\fr{\log{d_i}}{\log{z}}$, $B=\fr{\vp(W) \log{z}}{W}$. Then with $S(z)$ as defined in (\ref{support}), we have
\begin{equation*}
\begin{split}
\sum\lm_{\st{l_i \cg 0 \md{d_i}\\(l_1,l_2) \in S(z)}} f_1(l_1)f_2(l_2)
P\left( \fr{\log{l_1}}{\log{z}},\fr{\log{l_2}}{\log{z}} \right)&=B^{k_1+k_2} 
\fr{f_1(d_1)}{\ov{f}(d_1)}\fr{f_2(d_2)}{\ov{f}(d_2)}c(W,f) \\
& \quad \times \left( \iint\lm_{T_{s_1,s_2}} P(t_1,t_2)\fr{(t_1-s_1)^{k_1-1}}{(k_1-1)!}
\fr{(t_2-s_2)^{k_2-1}}{(k_2-1)!} \, dt_2 \, dt_1\right)\\
&\quad + O(\log^{k_1+k_2-1}z)
\end{split}
\end{equation*}
\end{cor}
\begin{proof}
First of all, let us rephrase the conditions $l_1 \cg 0(d_1)$, $l_2 \cg 0(d_2)$ and $(l_1,l_2) \in S(z)$. We know from Lemma \ref{etacn} that
$$ (l_1,l_2) \in S(z) \iff \mu^2(l_1l_2W)=1,  \ \text{and} \  
l_2 \leq z^{\eta\left( \fr{\log{l_1}}{\log{z}} \right)} $$
From now on, we let $ \ds t_i=\fr{\log{l_i}}{\log{z}}$, for $i=1,2$. We can now write the given summation as
$$\sum\lm_{\st{l_1 \cg 0(d_1)\\l_2 \cg 0(d_2)\\(l_1,l_2) \in S(z)}}=\sum\lm_{\st{l_1 \cg 0(d_1)\\(l_1,d_2) \in S(z)}} \sum\lm_{\st{l_2 \cg 0(d_2)\\(l_2,l_1) \in S(z)}}=\sum\lm_{\st{l_1 \cg 0(d_1)\\l_1 \leq z^{\eta(s_2)}\\(l_1,d_2W)=1}} 
\sum\lm_{\st{l_2 \cg 0(d_2)\\l_2 \leq z^{\eta(t_1)}\\(l_2,l_1W)=1}} $$
This given summation above is set in such a way that one can directly apply Theorem \ref{estcong} to the inner sum. We therefore have
\begin{equation}
\label{explain}
\begin{split}
& \quad \sum\lm_{\st{l_i \cg 0(d_i)\\(l_1,l_2) \in S(z)}}f_1(l_1)f_2(l_2)P\left( \fr{\log{l_1}}{\log{z}},\fr{\log{l_2}}{\log{z}} \right)=\sum\lm_{\st{l_1 \cg 0(d_1)\\l_1 \leq z^{\eta(s_2)}\\(l_1,d_2W)=1}}f_1(l_1)\sum\lm_{\st{l_2 \cg 0(d_2)\\l_2 \leq z^{\eta(t_1)}\\(l_2,l_1W)=1}}f_2(l_2)P\left( \fr{\log{l_1}}{\log{z}},\fr{\log{l_2}}{\log{z}} \right)\\
&=\fr{f_2(d_2)}{\ov{f_2}(d_2)}\sum\lm_{\st{l_1 \cg 0(d_1)\\l_1 \leq z^{\eta(s_2)}\\(l_1,d_2W)=1}}f_1(l_1) \left( \fr{\vp(Wl_1)\log{z}}{Wl_1} \right)^{k_2} c(Wl_1,f_2) 
\int\lm_{s_2}^{\eta\left(\fr{\log{l_1}}{\log{z}}\right)} P\left( \fr{\log{l_1}}{\log{z}},t_2 \right) \fr{(t_2-s_2)^{k_2-1}}{(k_2-1)!} \, dt_2 \\
&\quad+ O\left((\log z)^{k_2-1}  \sum\lm_{l_1 \leq z}f_1(l_1) \right)
\end{split}
\end{equation}
We can now write 
\begin{equation*}
\begin{split} 
c(Wl_1,f_2)&=\prod\lm_{p \nmid Wl_1}\left( 1-\fr{1}{p} \right)^{k_2} \ov{f_2}(p)=
 \left(\fr{l_1}{\vp(l_1)}\right)^k \fr{1}{\ov{f_2}(l_1)} \prod\lm_{p \nmid W}\left( 1-\fr{1}{p} \right)^{k_2} \ov{f_2}(p)\\
&=\left(\fr{l_1}{\vp(l_1)}\right)^k \fr{1}{\ov{f_2}(l_1)}c(W,f_2)
\end{split}
\end{equation*}
Therefore, substituting this expression for $c(Wl_1,f_2)$, the main term in (\ref{explain}) may be written as
\begin{equation*}
\begin{split}
&= \fr{f_2(d_2)}{\ov{f_2}(d_2)} \left( \fr{\vp(W) \log{z}}{W} \right)^{k_2} c(W,f_2)
\sum\lm_{\st{l_1 \cg 0(d_1)\\l_1 \leq z^{\eta(s_2)}\\(l_1,d_2W)=1}} \fr{f_1(l_1)}{\ov{f_2}(l_1)}\int\lm_{s_2}^{\eta\left(\fr{\log{l_1}}{\log{z}}\right)} P\left( \fr{\log{l_1}}{\log{z}},t_2 \right) \fr{(t_2-s_2)^{k_2-1}}{(k_2-1)!} \, dt_2\\
&= (f_1/\ov{f_1})(d_1)\fr{(f_2 / \ov{f_2})(d_2)}{\ov{(f_2/\ov{f_2})}(d_2)} 
\left( \fr{\vp(W)\log{z}}{W} \right)^{k_1+k_2} c(W,f_2) c\left(W,\fr{f_1}{\ov{f_2}} \right) \\
& \quad \times \int\lm_{s_1}^{\eta(s_2)} \int\lm_{s_2}^{\eta(t_1)} P(t_1,t_2) 
\fr{(t_1-s_1)^{k_1-1}}{(k_1-1)!} \fr{(t_2-s_2)^{k_2-1}}{(k_2-1)!} \, dt_2 \, dt_1
\end{split}
\end{equation*}
 Since $$ (f_1/\ov{f_1})(d_1)\fr{(f_2 / \ov{f_2})(d_2)}{\ov{(f_2/\ov{f_2})}(d_2)}=\fr{f_1(d_1)}{\ov{f}(d_1)}\fr{f_2(d_2)}{\ov{f}(d_2)} \quad \text{and} \quad c(W,f_2) c\left(W, \fr{f_1}{\ov{f_2}} \right)=c(W,f_1*f_2)=c(W,f)$$ 
we obtain the desired main term. Moreover, since $f_1 \in \Og_{k_1}$, it follows that the error term in (\ref{explain}) is $O(\log^{k_1+k_2-1} z)$. This completes the proof.
\end{proof}
\bigskip
\section{Estimates on Divisor sums}
The computation of $S_2$ transforms into certain divisor sums in arithmetic progressions. So, in this section we shall obtain asymptotic estimates for the sum
\begin{equation} \label{Dma} D_{m,a}(x)=\sum\lm_{\st{n \leq x\\n \cg a \md{m}}}\tau(n) \end{equation}
where $m$ is an even squarefree integer and $(a,m)=1$. To do so, we give an expression for $D_{m,a}(x)$ in terms of Exponential and Kloosterman sums (See Proposition \ref{idty}) and later use the Weil's bound. 
Using the convolution method, we can the use this expression of $D_{m,a}(x)$ to give asymptotic formula for the sum 
$$\sum\lm_{\st{n \leq x\\n \cg a \md{m}}}2^{\Og(n)}$$
These results can be found in Selberg's manuscript \cite[Pg 234-237]{AS}
\medskip
\begin{defn}
For any positive integer $m$, define
\begin{equation}
\label{Am}
A_m(x)=\sum\lm_{\st{n \leq x\\(n,m)=1}} \tau(n)
\end{equation}
\end{defn}
\medskip
The next Proposition gives us an expression for $A_m(x)$.
\begin{prop}
\label{Amest}
The following estimate holds for any squarefree positive integer $m \leq x$
$$A_m(x)=x\fr{\vp^2(m)}{m^2} \left( \log{x}+c + 2\sum\lm_{p \mid m}\fr{\log{p}}{p-1} \right) + O_{\eps}\left(x^{1/2} \sg_{-1/2}^2(m) \right)$$
\end{prop}
\begin{proof}
Define a function $\tau_m$ by
$$ \tau_m(n)=\begin{cases} \tau(n) & \text{if} \ (n,m)=1 \\ 0 & \text{otherwise} \end{cases} $$
Then the Dirichlet series for $\tau_m$ is
\begin{equation}
\label{Fm}
F_m(s)=\sum\lm_{\st{(n,m)=1}} \fr{\tau(n)}{n^s}=\prod\lm_{p \nmid m} \left( 1+\fr{2}{p^s} + \fr{3}{p^{2s}} + \dots \right)
=\zeta^2(s)\prod\lm_{p \mid m}\left( 1-p^{-s}\right)^2=\zeta^2(s)G_m(s)
\end{equation}
Let $g_m$ be the function defined by the Dirichlet series $G_m(s)$. Since $\zeta^2(s)$ is the Dirichlet series of $\tau(m)$, it follows from (\ref{Fm}) that $$\tau_m=\tau*g_m$$
Moreover, $g_m$ is given by
$$ g_m(n)=\sum\lm_{\st{a,b \mid m\\ab=n}} \mu(a) \mu(b) $$
Using the convolution method, we therefore have
\begin{equation}
\begin{split}
\sum\lm_{n \leq x} \tau_m(n)&=\sum\lm_{n \leq x} \sum\lm_{d \mid n} \tau\left(\fr{n}{d}\right) g_m(d)
=\sum\lm_{d \leq x} g_m(d) \sum\lm_{n \leq \fr{x}{d}} \tau(n)\\
&=\sum\lm_{d \leq x} g_m(d) 
\left[\fr{x(\log{x}-\log{d}+2\gm-1)}{d} + O\left(\fr{x^{1/2}}{d^{1/2}}\right)\right]\\
&=x(\log{x}+2\gm-1)\left( \sum\lm_{d \leq x} \fr{g_m(d)}{d} \right)-x\sum\lm_{d \leq x}\fr{g_m(d)\log{d}}{d}+O\left(x^{1/2} \sum\lm_{d \leq x}\fr{g_m(d)}{d^{1/2}}  \right)
\end{split}
\label{taumain}
\end{equation}

Now,
\begin{equation*}\sum\lm_{d}\fr{g_m(d)}{d}=G_m(1)=\fr{\vp^2(m)}{m^2}  \end{equation*}
We therefore have,
\begin{equation}
  \label{t1}
  \sum\lm_{d \leq x} \fr{g_m(d)}{d}=\fr{\vp^2(m)}{m^2} + O\left( \sum\lm_{d>x} \fr{g_m(d)}{d} \right)=\fr{\vp^2(m)}{m^2}+O\left( \fr{\tau^2(m)}{x} \right)
\end{equation}

Also,
\begin{equation*}
\sum\lm_{d}\fr{g_m(d)\log{d}}{d}=-G_m'(1)=-2G_m(1) \sum\lm_{p \mid m}\fr{\log p}{p-1}=-2\fr{\vp^2(m)}{m^2} \sum\lm_{p \mid m}\fr{\log p}{p-1}
\end{equation*}
Therefore,
\begin{equation}
 \label{t2}
\begin{split}
 \sum\lm_{d \leq x} \fr{g_m(d) \log d}{d}&=\sum\lm_d \fr{g_m(d) \log d}{d} + O\left( \sum\lm_{d>x} \fr{g_m(d) \log d}{d} \right)\\
&=-2\fr{\vp^2(m)}{m^2} \sum\lm_{p \mid m}\fr{\log p}{p-1} + O\left( \fr{\tau^2(m) \log x}{x} \right)
\end{split}
\end{equation}
and the error term in (\ref{taumain}) is 
\begin{equation}
\label{et}
\ll x^{1/2} \sum\lm_{d \leq x}\fr{g_m(d)}{d^{1/2}}=x^{1/2}\sum\lm_{d \leq x}\fr{1}{d^{1/2}}\sum\lm_{\st{a,b \mid m}}\ab{\mu(a)}\ab{\mu(b)}=x^{1/2}\sum\lm_{a,b \mid m}\fr{1}{(ab)^{1/2}}=x^{1/2} \sg_{-1/2}^2(m)
\end{equation}
Substituting the relations (\ref{t1}), (\ref{t2}) and (\ref{et}) back into (\ref{taumain}), we obtain the desired result.
\end{proof}
\medskip
Next, we give a relation of Divisor sums $D_{m,a}(x)$ and $A_m(x)$ (See \ref{Dma} and \ref{Am}) in terms of Exponential and Kloostermann sums.
\medskip
\begin{ntn}
By $\ds e(\al)$, we shall mean $\ds \exp\left( 2\pi i\al \right) $
\end{ntn}
\medskip
\begin{defn}[Kloosterman sums]
We define
\begin{equation}
\label{kloos}
S(a,b,m)=\sum\lm_{h\bar{h} \cg 1\md{m}}e\left(\fr{ah+b\bar{h}}{m}\right)
\end{equation} 
A sum of this type is called a \textbf{Kloosterman sum}.
\end{defn}
\medskip
We now state a famous result due to A. Weil.
\begin{thm}[A. Weil]
\label{aweil}
 \begin{equation*}
  |S(a,b,m)| \leq m^{1/2} \tau(m)(a,b,m)^{1/2}
 \end{equation*}
 \end{thm}
\medskip
\begin{defn}
For $\ds |\al| \leq \fr{1}{2}$, define
\begin{equation*} S_t(\al)=\sum\lm_{1 \leq n \leq t}e(n\al) \end{equation*}
\end{defn}
\begin{lem}
\label{mod1}
\begin{equation*}
 |S_t(\al)| \leq \fr{1}{2|\al|}
\end{equation*}
 \end{lem}
\vskip 0.4in 
\begin{prop}
\label{idty}
Let $m$ be any positive integer and $(a,m)=1$. Then we have 
\begin{equation}
\label{d1x}
\bigintssss\lm_{1}^x D_{m,a}(t)\fr{dt}{t}=\sum\lm_{\st{n \leq x\\n \cg a \md{m}}}\tau(n)\log{\left(\fr{x}{n}\right)}=
\fr{1}{m^2}\sum\lm_{-\fr{m}{2} < c,d \leq \fr{m}{2}}S(ac,d,m)\bigintssss\lm_{1}^x S_t\left(\fr{c}{m}\right)S_{x/t}\left(\fr{d}{m} \right)\fr{dt}{t}
\end{equation}
\begin{equation}
\label{a1x}
\bigintssss\lm_{1}^x A_{m}(t)\fr{dt}{t}=\sum\lm_{\st{n \leq x\\(n,m)=1}}\tau(n)\log{\left(\fr{x}{n}\right)}=
\fr{1}{m^2}\sum\lm_{-\fr{m}{2} < c,d \leq \fr{m}{2}}S(ac,0,m)S(d,0,m) \bigintssss\lm_{1}^x S_t\left(\fr{c}{m}\right)S_{x/t}\left(\fr{d}{m} \right)\fr{dt}{t}
\end{equation}
\end{prop}

\begin{proof}
First we rewrite the $RHS$ of $(\ref{d1x})$ in the form
\begin{equation} \label{intg} \fr{1}{m^2}\int\lm_{1}^x \left[ \sum\lm_{-\fr{m}{2} < c,d \leq \fr{m}{2}}S(ac,d,m)S_t\left(\fr{c}{m}\right)S_{x/t}\left(\fr{d}{m} \right)\right]\fr{dt}{t} \end{equation}
The sum inside the integral above is
\begin{equation}
\begin{split} 
\label{temp}
&\quad\sum\lm_{-\fr{m}{2} < c,d \leq \fr{m}{2}}S(ac,d,m)S_t\left(\fr{c}{m}\right)S_{x/t}\left(\fr{d}{m}\right) \\
&=\sum\lm_{-\fr{m}{2} < c,d \leq \fr{m}{2}}\sum\lm_{h \in Z_m^*}e\left(\fr{ach+d\bar{h}}{m}\right)
\sum\lm_{\st{r \leq t\\s \leq x/t}}e\left( \fr{rc+ds}{m} \right)\\
&=\sum\lm_{\st{r \leq t\\s \leq x/t}}\sum\lm_{h \in Z_m^*} 
\left( \sum\lm_{-\fr{m}{2} < c \leq \fr{m}{2}}e\left( \fr{c(ah+r)}{m} \right) \right)  
\left( \sum\lm_{-\fr{m}{2} < d \leq \fr{m}{2}}e\left( \fr{d(\bar{h}+s)}{m} \right) \right)
\end{split}
\end{equation}
Recalling the fact that $$\sum\lm_{b \md{m}}e\left(\fr{bx}{m}\right)=
\begin{cases} m & \text{if} \ x \cg 0 \md{m} \\ 0 & \text{otherwise} \end{cases}$$ 
equation $(\ref{temp})$ becomes
\begin{equation*}
\begin{split}
&\quad m^2\sum\lm_{\st{r\leq t\\s\leq x/t\\rs \cg a \md{m}}}
\sum\lm_{\st{h \in Z_m^*\\ah\cg -r \md{m}\\ \bar{h} \cg -s \md{m}}}1
=m^2\sum\lm_{\st{r \leq t\\s \leq x/t\\rs \cg a \md{m}}}1
\end{split}
\end{equation*}
The cardinality of $\{ h \in Z_m^* : ah \cg -r \md{m},\ \bar{h} \cg -s \md{m}\}$ is $1$ since $ha \cg -r \md{m}$ implies that $h \cg -r \bar{a} \md{m}$ and this forces a unique choice for $h \md{m}$.\\
Therefore, $(\ref{intg})$ becomes
\begin{equation*}
\quad\int\lm_{1}^x \left(\sum\lm_{\st{r \leq t\\s \leq x/t\\rs \cg a \md{m}}}1 \right)\fr{dt}{t}
=\sum\lm_{\st{rs \leq x\\rs \cg a \md{m}}}\log{\left( \fr{x}{rs} \right)}
=\sum\lm_{\st{n \leq x\\n \cg a \md{m}}}\tau(n)\log{\left( \fr{x}{n} \right)}
\end{equation*}
This proves $(\ref{d1x})$. \\
To prove the next part, we rewrite the $RHS$ of $(\ref{a1x})$ as
\begin{equation}
\label{intg2}
\fr{1}{m^2}\int\lm_{1}^x\left[\sum\lm_{-\fr{m}{2} < c,d \leq \fr{m}{2}}S(ac,0,m)S(d,0,m) 
S_t\left(\fr{c}{m}\right)S_{x/t}\left(\fr{d}{m} \right) \right]\fr{dt}{t}
\end{equation}

Again, the sum inside the integral above is
\begin{equation} 
\label{temp2}
\begin{split}
&\quad \sum\lm_{-\fr{m}{2} < c,d \leq \fr{m}{2}}S(ac,0,m)S(d,0,m)S_t\left(\fr{c}{m}\right)S_{x/t}\left( \fr{d}{m} \right)\\
&=\sum\lm_{-\fr{m}{2} < c,d \leq \fr{m}{2}}\sum\lm_{h_1,h_2 \in Z_m^*}e\left( \fr{ach_1 + dh_2}{m} \right) 
\left( \sum\lm_{\st{r \leq t\\s \leq x/t}} e\left( \fr{rc+sd}{m} \right) \right)\\
&=\sum\lm_{\st{r \leq t\\s \leq x/t}}\sum\lm_{h_1,h_2 \in Z_m^*} \left( \sum\lm_{-\fr{m}{2} < c \leq \fr{m}{2}}
e\left( \fr{c(ah_1+r)}{m} \right) \right) \left( \sum\lm_{-\fr{m}{2} < d \leq \fr{m}{2}} e\left( \fr{d(h_2+s)}{m} \right) \right)  
\end{split}
\end{equation}
Again, recalling the fact that 
$$\sum\lm_{b(\text{mod } m)}e\left(\fr{bx}{m}\right)=
\begin{cases} m & \text{if} \ x \cg 0 \md{m} \\ 0 & \text{otherwise} \end{cases}$$equation $(\ref{temp2})$ becomes
\begin{equation*}
\begin{split}
&\quad m^2\sum\lm_{\st{r \leq t\\s \leq x/t}}
\sum\lm_{\st{h_1,h_2 \in Z_m^*\\ah_1 \cg -r \md{m}\\h_2 \cg -s \md{m}}}1=m^2\sum\lm_{\st{r \leq t\\s \leq x/t\\(rs,m)=1}}1
\end{split}
\end{equation*}
Note that $|\{ (h_1,h_2) \in Z_m^* :  ah_1 \cg -r \md{m}, \ h_2 \cg -s \md{m}\}|$ is $1$ since $ah_1 \cg -r \md{m}$ implies that $h_1 \cg -r\bar{a} \md{m}$ for which there is only one solution. Similarly, there is a unique choice for $h_2$. 
Hence there is exactly one choice each for $h_1$ and $h_2$ modulo $m$. \\
Therefore, $(\ref{intg2})$ becomes
\begin{equation*}k
\int\lm_{1}^x\left[ \sum\lm_{\st{r \leq t\\s \leq x/t\\(rs,m)=1}}1\right]\fr{dt}{t}
=\sum\lm_{\st{rs \leq x\\(rs,m)=1}}\log{\left( \fr{x}{rs} \right)}
=\sum\lm_{\st{n \leq x\\(n,m)=1}}\tau(n)\log{\left( \fr{x}{n} \right)}
\end{equation*}
This completes the proof of the Proposition $\ref{idty}$.
\end{proof}
\medskip
\begin{lem}
Let $m$ be a squarefree positive integer. Then
\label{est}
\begin{equation}
\label{str}
\left| S(a,b,m)-\fr{S(a,0,m)(b,0,m)}{\vp(m)}\right|<2m^{1/2}\tau(m)(a,b,m)^{1/2} 
\end{equation}
When either of $a$ or $b$ is $\cg 0\md{m}$, the above difference is $0$.
\end{lem}
\begin{proof}
First, note that when either of $a$ or $b$ is $\cg 0 \md{m}$, say $b \cg 0\md{m}$, then the quantity in the LHS of (\ref{str}) is 
$$ \left| S(a,0,m) - \fr{S(a,0,m) S(0,0,m)}{\vp(m)} \right|=0 $$ since $S(0,0,m)=\vp(m)$.
\vskip 0.08in

When neither of $a$ or $b$ is divisible by $m$, we apply the trivial estimate. By Theorem $\ref{aweil}$
$$ \ds |S(a,b,m)|<m^{1/2}\tau(m)(a,b,m)^{1/2}$$
Secondly, since $m$ is squarefree, we have
$$S(a,0,m)=\sum\lm_{d \mid (a,m)}d\mu\left(\fr{m}{d}\right)=\mu(m)\sum\lm_{d \mid (a,m)}d\mu(d)=\mu(m)\mu(a,m)\vp(a,m)$$
So,
\begin{equation}|S(a,0,m)| = \vp(a,m) \end{equation}
Therefore 
\begin{equation*}
\begin{split}
\left|\fr{S(a,0,m)S(b,0,m)}{\vp(m)}\right| &= \fr{\vp(a,m)\vp(b,m)}{\vp(m)} \leq \fr{\vp(a,b,m)\vp[(a,m),(b,m)]}{\vp(m)} < \vp(a,b,m)\\
&<(a,b,m)^{1/2}m^{1/2}
\end{split}
\end{equation*}
This completes the proof.
\end{proof}
\bigskip
We shall make use of Propositions \ref{Amest} and \ref{idty} in order to obtain an expression for $D_{m,a}(x)$ with a decent error term.
\bigskip
\begin{thm}
\label{maint}
Let $m$ be a squarefree positive integer and let $(a,m)=1$. Then for any $\eps>0$, we have
\begin{equation}
\begin{split}
D_{m,a}(x)&=x\fr{\vp(m)}{m^2}\left( \log{x}+c+2\sum\lm_{p \mid m}\fr{\log{p}}{p-1}\right)
+ O_{\eps}\left( \fr{x^{1/2+\eps}}{m^{1/4}} \right)
\end{split}
\end{equation}
whenever $m \leq x^{2/3-\eps}$.
\end{thm}
\begin{proof}
Consider the difference 
\begin{equation}
\label{mdif}
\bigintssss\lm_{1}^x\left[ D_{m,a}(t) -\fr{1}{\vp(m)}A_{m}(t)\right]\fr{dt}{t}
\end{equation}
From Proposition $\ref{idty}$, the above expression (\ref{mdif}) is
\begin{equation}
\label{nice}
\fr{1}{m^2}\bigintssss\lm_{1}^x \left[\sum\lm_{-\fr{m}{2} < c,d \leq \fr{m}{2}}\left( S(ac,d,m)-\fr{S(ac,0,m)S(d,0,m)}{\vp(m)} \right) S_t\left( \fr{c}{m} \right)S_{x/t}\left( \fr{d}{m} \right) \right]\fr{dt}{t}
\end{equation}
By Lemma $\ref{est}$, the quantity inside the integral in (\ref{nice}) is zero when either $c \cg 0 \md{m}$  or $d \cg 0 \md{m}$. So we may assume that neither of $c,d$ is $0$. Then we have
\begin{equation*}
\left|S(ac,d,m)-\fr{S(ac,0,m)S(d,0,m)}{\vp(m)}\right|<2m^{1/2}\tau(m)(ac,d,m)^{1/2} \leq 2m^{1/2}\tau(m)(c,d,m)^{1/2}
\end{equation*}
Moreover, by Lemma \ref{mod1}, we have
$$\left| S_t\left( \fr{c}{m} \right) \right| \leq \fr{m}{2|c|} \quad \text{and} \quad \left| S_t\left( \fr{d}{m} \right) \right| \leq \fr{m}{2|d|}$$
Therefore, (\ref{nice}) is
\begin{equation}
\label{s_11}
\begin{split}
&\leq \fr{1}{m^2} \left[\sum\lm_{\st{-\fr{m}{2} < c,d \leq \fr{m}{2}\\c,d \neq 0}}\left( 2m^{1/2}\tau(m)(c,d,m)^{1/2} \right) \fr{m^2}{4|c||d|}\right]\bigintssss\lm_{1}^x\fr{dt}{t}\\
&\leq 2m^{1/2}\tau(m)\log{x}\sum\lm_{0 < c,d \leq \fr{m}{2}}\fr{(c,d,m)^{1/2}}{cd} \leq 2m^{1/2} \tau(m)\log{x}
\sum\lm_{0<c,d \leq \fr{m}{2}}\fr{(c,d)^{1/2}}{cd}
\end{split}
\end{equation}
Observe that
\begin{equation}
\begin{split}
\label{sum2}
\sum\lm_{0 < c,d \leq \fr{m}{2}}\fr{(c,d)^{1/2}}{cd}&=\sum\lm_{g \leq \fr{m}{2}}g^{1/2} \sum\lm_{\st{0<c,d \leq \fr{m}{2}\\(c,d)=g}}\fr{1}{cd}=\sum\lm_{g \leq \fr{m}{2}}g^{-3/2} \sum\lm_{0<c',d' \leq \fr{m}{2g^2}}\fr{1}{c'd'} \leq \zeta(3/2) \log^2{m}\\
\ll \log^2{x}
\end{split}
\end{equation}
Therefore, from $(\ref{s_11})$ and $(\ref{sum2})$, we get
\begin{equation*}
\left|\bigintssss\lm_{1}^x\left( D_{m,a}(t) -\fr{1}{\vp(m)}A_{m}(t)\right)\fr{dt}{t}\right| \ll m^{1/2}\tau(m)\log^3{x} \ll m^{1/2}x^{\eps}
\end{equation*}
We now let
$$ A(x)=D_{m,a}(x) \quad \text{and} \quad B(x)=\fr{1}{\vp(m)}A_m(x)$$
Let $C(x)=A(x)-B(x)$. We want to show that $C(x) \ll \fr{x^{1/2+\eps}}{m^{1/4}}$. Let
\begin{equation} \dl=m^{3/4} x^{1/2+\eps} \label{dl} \end{equation}
It is then clear that $\dl>m$ whenever $m \leq x^{2/3-\eps}$. Since $\int_1^x \fr{C(t)}{t} \, dt \ll m^{1/2}x^{\eps/2}$, it follows that 
\begin{equation} \label{xdl} \int_{x-\dl}^{x+\dl}\fr{C(t)}{t} \, dt \ll m^{1/2} x^{\eps/2}\end{equation}
Note that we have
$$ C(x)=\sum\lm_{\st{n \leq x\\(n,m)=1}} \tau(n) \left( 1_{n \cg a \md{m}} - \fr{1}{\vp(m)} \right) $$
Therefore, for any interval of length $m$, the jump in $C(x)$ is atmost $\fr{\dl}{m} x^{\eps}$. Now suppose that $C(x) \gg \fr{x^{1/2+\eps}}{m^{1/4}}$ for some $x$. 
Consider the neighbourhood $I_x=(x-\dl,x+\dl)$ of $x$. We then have for any $t \in I_x$ that
$$ C(t) \gg \fr{x^{1/2+\eps}}{m^{1/4}} + O \left( \fr{\dl}{m} x^{\eps} \right) $$
Under this choice of $\dl$ in (\ref{dl}), one immediately sees that $C(t) \gg \fr{x^{1/2+\eps}}{m^{1/4}}$. It then follows that
$$ \int_{x-\dl}^{x+\dl} \fr{C(t)}{t} \, dt \gg \fr{x^{1/2+\eps}}{m^{1/4}} \int_{x-\dl}^{x+\dl} \fr{dt}{t} = \fr{x^{1/2+\eps}}{m^{1/4}} \log \left( \fr{x+\dl}{x-\dl} \right) 
\gg \dl \fr{x^{-1/2+\eps}}{m^{1/4}}=m^{1/2} x^{2\eps}$$
This is a contradiction to (\ref{xdl}) and the proof is complete.
\end{proof}
\bigskip
\begin{thm}
Let $m=Wm'$ be an even squarefree integer, where $m'$ is odd and $(W,m')=1$. Suppose $(a,m)=1$. Then for any $\eps>0$
\begin{equation}
\sum\lm_{\st{n \leq x \\ n \cg a \md{m}}} 2^{\Og(n)} = x\fr{\vp(W)}{W^2}  \fr{c(W)}{g(m')} \left( \log{x}+c+2\sum\lm_{\st{p \mid Wm' \\ p>2}}\fr{\log{p}}{p-2} \right)
+O_{\eps}\left(\fr{x^{1/2+\eps}}{m^{1/4}}\right)
\end{equation}
whenever $m \leq x^{2/3-\eps}$. Note that here $\ds c(W)=\prod\lm_{p \nmid W}\fr{(p-1)^2}{p(p-2)}$ and $\ds g(p)=\fr{p(p-1)}{p-2}$.
\end{thm}
\begin{proof}
We shall make use of the Convolution method. We write
$$ 2^{\Og(n)}=\sum\lm_{d \mid n} a_d\tau\left( \fr{n}{d} \right) $$
Then the Dirichlet series of $a_n$ is given by
$$ A(s)= \prod\lm_{p} \fr{(1-p^{-s})^2}{(1-2p^{-s})}=\prod\lm_p \left( 1+\fr{1}{p^s(p^s-2)} \right)=\zeta(2s)B(s)$$
It then follows that $a_n \geq 0$, for all $n$. \\
Note here that $B(s)$ is convergent for $\sg>1/3$. Observe that $A(s)$ has a simple pole at $s=1/2$. This means that
\begin{equation} \sum\lm_{n \leq x}\fr{a_n}{n^{1/2}}=O(\log{x}) \label{1/2}\end{equation}
We have
\begin{equation*}
\begin{split}
\sum\lm_{\st{n \leq x\\n \cg a\md{m}}}2^{\Og(n)}&=\sum\lm_{\st{n \leq x\\n \cg a\md{m}}} \sum\lm_{d \mid n} a_d \tau\left( \fr{n}{d} \right)=\sum\lm_{\st{d \leq x\\(d,m)=1}}a_d 
\sum\lm_{\st{n \leq x\\n \cg a\md{m}\\n \cg 0\md{d}}}\tau\left( \fr{n}{d} \right)\\
&=\sum\lm_{\st{d \leq x\\(d,m)=1}}a_d \sum\lm_{\st{n \leq \fr{x}{d}\\n \cg a'\md{m}}} \tau(n)=\sum\lm_{\st{d \leq x\\(d,m)=1}}a_d \ D_{m,a'}\left(\fr{x}{d}\right)\\
&=x\sum\lm_{\st{d \leq x\\(d,m)=1}} \fr{a_d}{d} \fr{\vp(m)}{m^2}\left( \log{x}-\log{d}+c+2\sum\lm_{p \mid m} \fr{\log{p}}{p-1} \right)+ O_{\eps} \left(\fr{x^{1/2+\eps}}{m^{1/4}} 
\sum\lm_{\st{d \leq x}} \fr{a_d}{d^{1/2}}  \right)
\end{split}
\end{equation*}
From (\ref{1/2}), the error term above becomes $O\left( \fr{x^{1/2+\eps}}{m^{1/4}} \log{x}\right)=O\left(\fr{x^{1/2+\eps}}{m^{1/4}}\right)$. The main term is
\begin{equation}
\label{mt}
x\fr{\vp(m)}{m^2}\left(\log{x}+c+2\sum\lm_{p \mid m}\fr{\log{p}}{p-1}\right) \sum\lm_{\st{d \leq x\\(d,m)=1}} \fr{a_d}{d} - x\fr{\vp(m)}{m^2} \sum\lm_{\st{d \leq x\\(d,m)=1}} \fr{a_d}{d} \log{d}
\end{equation}
Now,
\begin{equation}
\label{ad/d}
\sum\lm_{\st{d \leq x\\(d,m)=1}} \fr{a_d}{d}=\sum\lm_{(d,m)=1}\fr{a_d}{d} + O \left( \fr{\log x}{\sr{x}} \right)=\prod\lm_{p \nmid m}\fr{(p-1)^2}{p(p-2)} + O\left( \fr{\log x}{\sr{x}} \right)
\end{equation}
Moreover, since 
$$ \sum\lm_{(d,m)=1}\fr{a_d}{d^s}=\prod\lm_{p \nmid m}\left( 1+\fr{1}{p^s(p^s-2)} \right)$$
we can take logarithmic derivatives to obtain
\begin{equation*}
\sum\lm_{(d,m)=1} \fr{a_d}{d^s} \log{d}=-\prod\lm_{p \nmid m}\left( 1+\fr{1}{p^s(p^s-2)} \right) 
\left( -2\sum\lm_{p \nmid m} \fr{\log{p}}{(p^s-1)(p^s-2)} \right)
\end{equation*}
Substitute $s=1$ to obtain
\begin{equation}
\label{adlogd/d}
\sum\lm_{(d,m)=1}\fr{a_d}{d}\log{d}=\prod\lm_{p \nmid m}\fr{(p-1)^2}{p(p-2)} \left( 2\sum\lm_{p \nmid m} \fr{\log{p}}{(p-1)(p-2)} \right)=\prod\lm_{p \nmid m}\fr{(p-1)^2}{p(p-2)} \left( C_0-2\sum\lm_{\st{p \mid m\\p>2}} \fr{\log{p}}{(p-1)(p-2)} \right)
\end{equation}
Substituting the equations (\ref{ad/d}) and (\ref{adlogd/d}) into (\ref{mt}), we obtain
$$ \sum\lm_{\st{n \leq x\\n \cg a\md{m}}}2^{\Og(n)}=x\fr{\vp(m)}{m^2} \prod\lm_{p \nmid m}\fr{(p-1)^2}{p(p-2)} \left( \log{x} + c + 2\sum\lm_{\st{p \mid m\\p>2}} \fr{\log{p}}{p-2} \right) + O_{\eps}\left( \fr{x^{1/2+\eps}}{m^{1/4}} \right) $$
The proof is completed with the following observation
$$ \fr{\vp(m)}{m^2} c(m)=\prod\lm_{p \mid m} \fr{(p-1)}{p^2} \prod\lm_{p \nmid m} \fr{(p-1)^2}{p(p-2)}=\fr{\vp(W)}{W^2}
\fr{1}{g(m')} \prod\lm_{p \nmid W} \fr{(p-1)^2}{p(p-2)}=\fr{\vp(W)}{W^2} \fr{c(W)}{g(m')} $$
\end{proof}
We have an immediate Corollary from the above Theorem.
\begin{cor}
\label{2^Og}
Let $m=Wm'$ be an even squarefree integer, where $m'$ is odd and $(W,m')=1$. Suppose $(a,m)=1$. Then for any $\eps>0$ and $m \leq x^{2/3-\eps}$, we have
\begin{equation*}
\sum\lm_{\st{n \sim x \\ n \cg a \md{m}}} 2^{\Og(n)} = x\fr{\vp(W)}{W^2} \fr{c(W)}{g(m')} \left( \log{x}+c'+
2\sum\lm_{\st{p \mid Wm' \\ p>2}}\fr{\log{p}}{p-2} \right)+O_{\eps}\left(\fr{x^{1/2+\eps}}{m^{1/4}}\right)
\end{equation*}
where $n \sim x$ means that $x<n \leq 2x$.
\end{cor}

\section{The Modified Sieve}
We now use two-dimensional sieve weights to Selberg's approximation of the twin prime problem and reduce the value of $\lb$. We shall first provide asymptotic expressions for $S_1$ and $S_2$ in Proposition \ref{asymptote}. The later part of this section will be devoted to proving this Proposition. Once these proofs are completed, we make use of a sage program to optimise the sieve weights and determine a suitable value of $\lb$.
\medskip
Following Maynard\cite{JM}, we set 
$$D=\log\log\log{x} \quad \text{and} \quad W=\prod\lm_{p \leq D} p $$
This is done in order to avoid complications in our calculations. \\ \vskip 0.01in
Let $h$ be an even number and $v_0$ be chosen such that $(n,W)=(n+h,W)=1$. We consider the sum
\begin{equation} \sum\lm_{\st{n \sim x\\n \cg v_0 \md{W}}} \left( 1-\frac{2^{\Og(n)}+2^{\Og(n+h)}}{\lb} \right)\left(\sum\lm_{\st{d_1 \mid n \\ d_2 \mid n+h}} \lb_{d_1,d_2}\right)^2 =S_1 -S_2/\lb  \label{thesum}\end{equation}
where
\begin{equation}
\label{s11}
S_1=\sum\lm_{\st{n \sim x\\n \cg v_0 \md{W}}} \left( \sum\lm_{\st{d_1 \mid n \\d_2 \mid n+h}}\lb_{d_1,d_2} \right)^2
\end{equation}
and
\begin{equation}
\label{s22}
S_2=\sum\lm_{\st{n \sim x\\n \cg v_0 \md{W}}} \left( 2^{\Og(n)} + 2^{\Og(n+h)} \right) \left( \sum\lm_{\st{d_1 \mid n\\d_2 \mid n+h}}\lb_{d_1,d_2} \right)^2
\end{equation}
\bigskip
Let
\begin{equation}
\label{S(z)}
S(z)=\{(d_1,d_2):\mu^2(d_1d_2W)=1, \ \max\{d_1d_2^{2/3}, d_1^{2/3}d_2 \} \leq z\}
\end{equation}
\medskip
This is going to be the support of the sieve weights $\lb_{d_1,d_2}$. The reason for this choice will soon be clear. We will choose
\begin{equation}  \label{zchoice} z=x^{1/3-\eps} \end{equation}
\begin{ntn}
\label{ij}
Until the evaluations of $S_1$ and $S_2$ are complete, the set $\{i,j\}$ will be a permutation of the set $\{1,2\}$. So, if we write $s_{ij} \mid (d_i,l_j)$, it means that both the conditions $s_{12} \mid (d_1,l_2)$ and $s_{21} \mid (d_2,l_1)$ hold. Moreover, if we write $d_i \cg 0\md{l_i}$, it will mean that both the relations $d_1 \cg 0 \md{l_1}$ and $d_2 \cg 0 \md{l_2}$ hold.
\end{ntn}
Instead of making a choice for $\lb_{d_1,d_2}$ directly, we define the quantity
\begin{equation}
D_{r_1,r_2}=f(r_1)f(r_2) \sum\lm_{\st{d_i \cg 0 \md{r_i}}}\fr{\lb_{d_1,d_2}}{f(d_1)f(d_2)}
\end{equation}
and make a choice for $D_{r_1,r_2}$. Note that $D_{r_1,r_2}$ vanishes whenever $(r_1,r_2)$ is not in $S(z)$. \vskip 0.04in
Here, $f$ is defined multiplicatively by
\begin{equation} \label{ff1} f(p)=p/2 \end{equation} 
We set
\begin{equation}
 D_{r_1,r_2}=\begin{dcases} \mu(r_1)\mu(r_2)P\left( \fr{\log{r_1}}{\log z},\fr{\log r_2}{\log z} \right) & (r_1,r_2) \in S(z) \\ 0 & \text{otherwise} \end{dcases}
\end{equation}
where $P:T \to \mb{R}$ is a bounded and \textbf{symmetric} differentiable function, with $T$ as defined in (\ref{Ts1s2}). \vskip 0.02in
The reason for choosing $P$ to be symmetric is that we want the $\lb_{d_1,d_2}$'s to be symmetric and this is equivalent to the $D_{r_1,r_2}$'s being symmetric. Morevover, choosing $D_{r_1,r_2}$ is equivalent 
to choosing $\lb_{d_1,d_2}$ due to the relation
\begin{equation} 
\begin{split} 
\fr{\lb_{d_1,d_2}}{f(d_1)f(d_2)}&=\mu(d_1)\mu(d_2)\sum\lm_{l_i \cg 0\md{d_i}} \fr{\mu(l_1)\mu(l_2)}{f(l_1)f(l_2)}D_{l_1,l_2}\\
&=\mu(d_1)\mu(d_2)\sum\lm_{\st{l_i \cg 0\md{d_i} \\ (l_1,l_2) \in S(z)}} \fr{\mu^2(l_1)\mu^2(l_2)}{f(l_1)f(l_2)}P\left( \fr{\log l_1}{\log z}, \fr{\log l_2}{\log z} \right)
\label{rel}
\end{split}
\end{equation}
\medskip
\begin{ntn}
We set \begin{equation} B =\fr{\vp(W) \log{z}}{W} \label{normalize}\end{equation}
\end{ntn}
\medskip
The following Proposition gives us the asymptotic expressions for $S_1$ and $S_2$ in terms of the function $P$.
\begin{prop}
\label{asymptote}
Let $\lb_{d_1,d_2}$ be as in (\ref{rel}) and let $\eta$ be as defined in Definition \ref{etadef}. Then with $z=x^{1/3-\eps}$
\begin{equation*}
\begin{split}
S_1 &= \fr{x}{W}  (1+o(1)) B^6 R_1(P)\\
S_2 &= \fr{2x}{W} (1+o(1)) B^6 R_2(P) 
\end{split}
\end{equation*}
where 
\begin{equation}
\begin{split}
\label{R1R2}
R_1(P) &= \iint\lm_{T \ \ } Q_1^2(s_1,s_2) \, ds_2 \, ds_1 \\
R_2(P) &= \iint\lm_{T \ \ } (s_1(3-s_1) Q_2^2(s_1,s_2) + 4 s_1 Q_1(s_1,s_2) Q_2(s_1,s_2)  ) \, ds_2 \, ds_1
\end{split}
\end{equation}
Here
\begin{equation}
\label{Q1Q2}
Q_1(s_1,s_2) = \iint\lm_{T_{s_1,s_2} \ } P(t_1,t_2) \, dt_2 \, dt_1  \quad \text{and} \quad Q_2(s_1,s_2) = \int\lm_{s_2}^{\eta(s_1)} P(s_1,t_2) \, dt_2
\end{equation}
\end{prop}

The remainder of this section will be devoted to proving Proposition \ref{asymptote}. Before we begin, we define two new quantities which will arise when computing $S_1$ and $S_2$. Let
\begin{equation}
\begin{split}
\label{BCr1r2}
B_{r_1,r_2} &= \vp(r_1) \vp(r_2) \sum\lm_{d_i \cg 0 \md{r_i}} \fr{\lb_{d_1,d_2}}{d_1 d_2}\\
C_{r_1,r_2} &= f_1(r_1) g_1(r_2) \sum\lm_{d_i \cg 0 \md{r_i}} \fr{\lb_{d_1,d_2}}{f(d_1) g(d_2)}
\end{split} 
\end{equation}
where $f=f_1*1$ is as defined in (\ref{ff1}) and $g$, $g_1$ are defined by
\begin{equation}
\label{gg1}
g(p)=\fr{p(p-1)}{p-2} \quad \text{and} \quad g_1*1=g
\end{equation}
\medskip
\begin{lem}
\label{LBCr1r2}
The following relations hold for any $(r_1,r_2) \in S(z)$
\begin{itemize}
\item[a)] $ \ds |\lb_{r_1,r_2}| \ll \log^4 z$ 
\item[b)] $ \ds B_{r_1,r_2}=\kp_1(W) B^2  \mu(r_1)\mu(r_2)h(r_1)h(r_2) \ Q_1\left( \fr{\log r_1}{\log z},\fr{\log r_2}{\log z} \right) + O(\log z)$ \\
where $Q_1$ is defined as in (\ref{Q1Q2}). Here 
$$\kp_1(W)=\prod\lm_{p \nmid W}\left(1-\fr{1}{p}\right)^2 \left(1+\fr{2}{p}\right) \quad \text{and} \quad h(p)=1-3/(p+2)$$ 
\item[c)] $ \ds C_{r_1,r_2}= B \ \mu(r_1)\mu(r_2)h_1(r_1)h_2(r_2) \ Q_2\left( \fr{\log r_1}{\log z},\fr{\log r_2}{\log z} \right) + O(1)$ \\
where $Q_2$ is as defined in (\ref{Q1Q2}). Here
$$h_1(p)=1-3/p+2/p^2 \quad \text{and} \quad h_2(p)=1-2/p+2/p^2$$
\end{itemize}
\end{lem}

\begin{proof}\text{ }
\begin{itemize}
\item[a)] From (\ref{rel}), we have
$$ \lb_{r_1,r_2}=\mu(r_1)\mu(r_2)f(r_1)f(r_2)\sum\lm_{\st{l_i \cg 0(r_i) \\ (l_1,l_2) \in S(z)}} \fr{\mu^2(l_1)\mu^2(l_2)}{f(l_1)f(l_2)}P\left( \fr{\log l_1}{\log z}, \fr{\log l_2}{\log z} \right) $$
Applying Corollary \ref{sumS} to the RHS of the above expression and writing $s_i=\fr{\log r_i}{\log z}$, we obtain
\begin{equation*}
\begin{split}
\lb_{r_1,r_2}&=c(W,f*f) B^4  \fr{\mu(r_1)\mu(r_2)}{\prod\lm_{p \mid r_1 r_2}(1+4/p)} \iint\lm_{T_{s_1,s_2}} P(t_1,t_2) (t_1-s_1)(t_2-s_2) \, dt_2 \, dt_1 + O(\log^3 z)\\
& \ll \log^4 z
\end{split}
\end{equation*}
Here $c(W,f*f)$ is as defined in (\ref{cmf}) and tends to $1$ as $x \to \infty$.

\item[b)]
From (\ref{BCr1r2}) and (\ref{rel}), we have
\begin{equation*}
\begin{split}
\fr{B_{r_1,r_2}}{\vp(r_1)\vp(r_2)} &= \sum\lm_{d_i \cg 0 \md{r_i}}\fr{\lb_{d_1,d_2}}{d_1d_2}=\sum\lm_{d_i \cg 0 \md{r_i}} \fr{\mu(d_1)\mu(d_2)}{\tau(d_1)\tau(d_2)} \sum\lm_{\st{l_i \cg 0 \md{d_i}\\(l_1,l_2) \in S(z)}} 
\fr{\mu^2(l_1) \mu^2(l_2)}{f(l_1)f(l_2)} P\left( \fr{\log l_1}{\log z}, \fr{\log l_2}{\log z} \right)\\
&= \sum\lm_{\st{l_i \cg 0 \md{r_i}\\(l_1,l_2) \in S(z)}}\fr{\mu^2(l_1) \mu^2(l_2)}{f(l_1)f(l_2)} P\left( \fr{\log l_1}{\log z}, \fr{\log l_2}{\log z} \right) 
\sum\lm_{r_i \mid d_i \mid l_i} \fr{\mu(d_1)\mu(d_2)}{\tau(d_1)\tau(d_2)}\\
&= \mu(r_1)\mu(r_2) \sum\lm_{\st{l_i \cg 0 \md{r_i}\\(l_1,l_2) \in S(z)}} \fr{\mu^2(l_1) \mu^2(l_2)}{l_1 l_2} P\left( \fr{\log l_1}{\log z}, \fr{\log l_2}{\log z} \right)
\end{split}
\end{equation*}
Applying Corollary \ref{sumS} to the last equality above functions $\mu^2/ id , \ \mu^2/id \in \Og_1$ , we obtain the desired result.
\item[c)] Again from (\ref{BCr1r2}) and (\ref{rel}), we have
\begin{equation}
\begin{split}
\label{Ctemp}
\fr{C_{r_1,r_2}}{f_1(r_1)g_1(r_2)} &= \sum\lm_{d_i \cg 0 \md{r_i}} \fr{\lb_{d_1,d_2}}{f(d_1)g(d_2)}\\
&= \sum\lm_{d_i \cg 0 \md{r_i}} \mu(d_1)\fr{\mu(d_2)f(d_2)}{g(d_2)} 
\sum\lm_{\st{l_i \cg 0\md{d_i}\\(l_1,l_2) \in S(z)}} \fr{\mu^2(l_1) \mu^2(l_2)}{l_1l_2}P\left( \fr{\log l_1}{\log z},\fr{\log l_2}{\log z} \right)\\
&=\sum\lm_{\st{l_i \cg 0\md{r_i}\\(l_1,l_2) \in S(z)}} \fr{\mu^2(l_1) \mu^2(l_2)}{l_1l_2}P\left( \fr{\log l_1}{\log z},\fr{\log l_2}{\log z} \right) \sum\lm_{r_i \mid d_i \mid l_i} \mu(d_1)\fr{\mu(d_2)f(d_2)}{g(d_2)}\\
&= \fr{\mu(r_1)}{f(r_1)} \fr{\mu(r_2) \vp(r_2)}{g(r_2)} \sum\lm_{\st{l_2 \cg 0 \md{r_2}\\ (r_1,l_2) \in S(z)} } \fr{\mu^2(l_2)}{\vp(l_2)}
\end{split}
\end{equation}
In the last equality in (\ref{Ctemp}), we note from Lemma \ref{etacn} that 
$$(r_1,l_2) \in S(z) \iff \mu^2(r_1l_2W)=1 \quad \text{and} \quad l_2 \leq z^{\eta\left( \fr{\log r_1}{\log z} \right)}$$ 
By applying Theorem \ref{estcong} to the sum in the last equality of (\ref{Ctemp}) with the function $\mu^2/\vp \in \Og_1$, we shall obtain the desired result.
\end{itemize}
\end{proof}

\bigskip
We now begin with the evaluation of $S_1$.
\subsection{Evaluation of \tps{$S_1$}{}}
From (\ref{s11}), we have
\begin{equation}
\label{1st}
\begin{split}
S_1 &=\sum\lm_{\st{n \sim x\\ n \cg v_0 \md{W}}} \left( \sum\lm_{\st{d_1 \mid n\\d_2 \mid n+h}}\lb_{d_1,d_2} \right)^2= \sum\lm_{\st{n \sim x\\n \cg v_0 \md{W}}}\sum\lm_{\st{d_1,l_1 \mid n \\ d_2,l_2 \mid n+h}} \lb_{d_1,d_2} \lb_{l_1,l_2}\\
&=\sum\lm_{\st{(d_i,l_j)=1}} \lb_{d_1,d_2} \lb_{l_1,l_2} \sum\lm_{\st{n \sim x \\ n \cg 0 \md{[d_1,l_1]} \\ n \cg -h \md{[d_2,l_2]}\\n \cg v_0 \md{W}}}1
\end{split}
\end{equation}
In the above sum, we have the conditions $d_1,l_1 \mid n$ and $d_2,l_2 \mid n+h$. We first choose $x$ large enough so that $D>h$. This ensures that $rad(h) \mid W$. It then follows that $(d_i,d_j)=(d_i,l_j)=1$ for $i \neq j$ and that $(d_1d_2l_1l_2,W)=1$. 
This is because if there is a prime $p$ dividing $(d_i,d_j)$ ( or $(d_i,l_j)$) for $i \neq j$, then $p$ must divide both $n$ and $n+h$ and therefore $p \mid h$. But since $rad(h) \mid W$, it follows that $p \mid W$. This is a contradiction because the numbers $d_i$ and $l_j$ are all coprime to $W$. Moreover, the conditions $(d_1,d_2)=1$ and $(l_1,l_2)=1$ can be dropped because they are already included in the definition of $S(z)$ (See (\ref{S(z)})), the support of the sieve weights $\lb_{d_1,d_2}$. So, we are only left with the conditions $(d_1,l_2)=(d_2,l_1)=1$.
Since the numbers $[d_1,l_1]$, $[d_2,l_2]$ and $W$ are pairwise coprime, the inner sum in the last equality of (\ref{1st}) becomes 
$$\fr{x}{W[d_1,l_1][d_2,l_2]} + O(1)$$
Therefore,
\begin{equation}
\label{mainabove}
\begin{split}
S_1&=\sum\lm_{\st{(d_i,l_j)=1}} \lb_{d_1,d_2} \lb_{l_1,l_2} \left( \frac{x}{W[d_1,l_1][d_2,l_2]} +O(1) \right)= \fr{x}{W}\sum\lm_{\st{(d_i,l_j)=1}} \frac{\lb_{d_1,d_2}\lb_{l_1,l_2}}{[d_1,l_1][d_2,l_2]}+ 
O\left(\sum\lm_{\st{d_i,l_i}}|\lb_{d_1,d_2}||\lb_{l_1,l_2}|\right)\\
&=M_1+E_{11}\\
\end{split}
\end{equation}
The error term in (\ref{mainabove}) is
\begin{equation} 
\label{E11} 
E_{11} \ll \left(\sum\lm_{d_1,d_2}|\lb_{d_1,d_2}|\right)^2 \ll \left( \sum\lm_{(d_1,d_2) \in S(z)} \log^4 z \right)^2
\ll |S(z)|^2 \log^8 z \ll z^{8/3} \log^8 z
\end{equation}
where we have used the bound $|\lb_{d_1,d_2}| \ll \log^4 z$ from Lemma \ref{LBCr1r2} and the fact that $|S(z)| \ll z^{4/3}$.\\ \vskip 0.02in
In the main term $M_1$, we can write $$ \fr{1}{[d_i,l_i]}=\fr{1}{d_il_i}\sum\lm_{r_i \mid (d_i,l_i)}\vp(r_i)$$
This gives
$$ M_1=\fr{x}{W}\sum\lm_{\st{(d_i,l_j)=1}}
\fr{\lb_{d_1,d_2}\lb_{l_1,l_2}}{d_1d_2 \ l_1l_2} \sum\lm_{\st{r_i \mid (d_i,l_i)}} \vp(r_1)\vp(r_2)$$
To get rid of the conditions $(d_1,l_2)=(d_2,l_1)=1$, we multiply a factor of 
$$\sum\lm_{\st{s_{ij} \mid (d_i,l_j)}} \mu(s_{12})\mu(s_{21})$$
This idea was used by Maynard in his preprint \cite{JM}. \vskip 0.01in
Now, $r_i \mid (d_i,l_i)$ and $s_{ij} \mid (d_i,l_j)$, for $i \neq j$. This implies that $(r_i,s_{ij})=(r_i,s_{ji})=1$ because if there was prime $p$ dividing $r_i$ and $s_{ij}$, then $p \mid (l_i,l_j)$. This is contradiction since $(l_1,l_2)=1$. Also note that both $r_i$ and $s_{ij}$ divide $d_i$. Since $(r_i,s_{ij})=1$, it follows that $r_i s_{ij} \mid d_i$. Similarly, we get $r_j s_{ji} \mid d_j$. Therefore it follows that both $(r_1s_{12},r_2s_{21})$ and $(r_1s_{21},r_2s_{12})$ are in $S(z)$. Summarising this, we have
\begin{equation}
\label{m1}
\begin{split}
M_1&= \fr{x}{W} \sum\lm_{r_i} \vp(r_1) \vp(r_2)\sum\lm_{\st{d_i,l_i \cg 0(r_i)}} \frac{\lb_{d_1,d_2}}{d_1d_2}\frac{\lb_{l_1,l_2}}{l_1l_2} \sum\lm_{\st{s_{ij} \mid (d_i,l_j)}}\mu(s_{12})\mu(s_{21})\\
&=\fr{x}{W}\sum\lm_{r_i} \vp(r_1) \vp(r_2)\sum\lm_{s_{ij}} \mu(s_{12})\mu(s_{21}) 
\left( \sum\lm_{\st{d_i \cg 0(r_is_{ij})}} \frac{\lb_{d_1,d_2}}{d_1d_2} \right)  
\left( \sum\lm_{\st{l_i \cg 0(r_is_{ji})}} \frac{\lb_{l_1,l_2}}{l_1l_2} \right)
\end{split}
\end{equation}
\begin{rem}
\label{neednot}
We need not write  the conditions $(r_1s_{12},r_2s_{21}) \in S(z)$ and $(r_1s_{21},r_2s_{12}) \in S(z)$ at every step because 
one of the two bracketed quantities in the last equality of (\ref{m1}) will vanish when any of these conditions do not hold (This is because the support of $\lb_{d_1,d_2}$ is $S(z)$). 
\end{rem}
Recalling the Definition of $B_{r_1,r_2}$ from (\ref{BCr1r2}), note that the two bracketed quantities in the last equality of (\ref{m1}) can be replaced by appropriate $B_{r_1,r_2}$'s i.e we obtain
\begin{equation}
\label{s1}
M_1=\fr{x}{W} \sum\lm_{r_i} \fr{\mu^2(r_1)}{\vp(r_1)} \fr{\mu^2(r_2)}{\vp(r_2)}
\sum\lm_{s_{ij}} \fr{\mu(s_{12})}{\vp^2(s_{12})}\fr{\mu(s_{21})}{\vp^2(s_{21})} B_{r_1s_{12},r_2s_{21}}B_{r_1 s_{21},r_2 s_{12}}
\end{equation} \\ \vskip 0.02in

We would now like to get rid of the numbers $s_{12}$ and $s_{21}$ from the above summation. Since $(s_{12},s_{21}) \in S(z)$, we have $(s_{12},W)=(s_{21},W)=1$. Therefore, either $s_{ij}=1$ or $s_{ij}>D$. The contribution to the main term $M_1$ of the above sum (\ref{s1}) when atleast one of $s_{ij}$ is greater than $D$ is 
\begin{equation*}
\begin{split}
 & \ll \fr{x}{W} \sum\lm_{\st{(s_{12},s_{21}) \in S(z)\\ s_{12}>D}}\fr{1}{\vp^2(s_{12})\vp^2(s_{21})}
\sum\lm_{\st{ (r_1,r_2) \in S(z)}} \fr{\mu^2(r_1)}{\vp(r_1)} \fr{\mu^2(r_1)}{\vp(r_2)} \log^4{z} \\
& \ll \fr{x \log^4 z}{W} \left( \sum\lm_{s_{12>D}} \fr{1}{\vp^2(s_{12})} \right) \left( \sum\lm_{s_{21}} \fr{1}{\vp^2(s_{21})} \right) 
\left( \sum\lm_{r_1, r_2 \leq z} \fr{1}{\vp(r)\vp(r_2)} \right) \ll \fr{x \log^6 z}{WD}
\end{split}
\end{equation*}
where we have used the estimate $B_{r_1,r_2} \ll \log^4 z$ from Lemma \ref{LBCr1r2}. \vskip 0.02in
We may therefore assume that $s_{12}=s_{21}=1$ with a cost of error term of 
\begin{equation} \label{E12} E_{12}=O \left( \fr{x\log^6{z}}{WD} \right) \end{equation} 
Therefore, one can now write
\begin{equation}
\label{s1temp}
S_1=\fr{x}{W}\sum\lm_{r_i}\fr{\mu^2(r_1)}{\vp(r_1)} \fr{\mu^2(r_2)}{\vp(r_2)} B^2_{r_1,r_2} + E_{11}+E_{12}
\end{equation}
Substituting the value of $B_{r_1,r_2}$ from Lemma $\ref{LBCr1r2}$, we obtain
\begin{equation}
\label{kpB}
M_1= \fr{x}{W}  \sum\lm_{\st{(r_1,r_2) \in S(z)}} \fr{\mu^2(r_1)}{\vp(r_1)}\fr{\mu^2(r_2)}{\vp(r_2)}
\left( \kp_1^2(W)  B^4 h^2(r_1)h^2(r_2) Q_1^2\left( \fr{\log{r_1}}{\log{z}},\fr{\log{r_2}}{\log{z}} \right) + O(\log^3 z) \right)
\end{equation}
Note that the inner sum in (\ref{kpB}) is precisely the type of sum estimated in Corollary \ref{sumS}. Applying Corollary \ref{sumS} with the functions $h^2/\vp, \ h^2/\vp \in \Og_1$ ( See Definition \ref{omegak}), we obtain 
\begin{equation}
M_1=\fr{x}{W} B^6 \iint\lm_{T \ \ }Q_1^2(s_1,s_2) \,ds_2 \,ds_1 + O\left( \fr{x \log^5 z}{W} \right)
\end{equation}
with $Q_1$ defined as in (\ref{Q1Q2}). \vskip 0.01in

When we choose $z=x^{1/3-\eps}$, we have 
$$E_{11} \ll x^{8/9-8 /3 \eps} \quad \text{and} \quad E_{12} \ll  \fr{x \log^6 z}{WD}$$
This gives the following final expression for $S_1$
\begin{equation}
\label{S1}
S_1=\fr{x}{W}(1+o(1))B^6 \ R_1(P)
\end{equation}
where 
\begin{equation}
\label{R1}
R_1(P)= \iint\lm_{T \ \ } Q_1^2(s_1,s_2) \,ds_2 \,ds_1 
\end{equation}

\bigskip
\subsection{Evaluation of \tps{$S_2$}{}}
We will use the same ideas here as we did while computing $S_1$. The conditions and observations given in the paragraph right after equation (\ref{1st}) will be implemented. \vskip 0.01in
From (\ref{s22}), we have 
\begin{equation}
\label{s2tent}
\begin{split}
S_2 &=\sum\lm_{\st{n \sim x\\n \cg v_0 \md{W}}} (2^{\Og(n)}+2^{\Og(n+h)}) \sum\lm_{\st{d_1,l_1 \mid n \\ d_2,l_2 \mid n+h}} 
\lb_{d_1,d_2} \lb_{l_1,l_2}=\sum\lm_{\st{(d_i,l_j)=1}} \lb_{d_1,d_2} \lb_{l_1,l_2}
\sum\lm_{\st{n \sim x \\ n \cg 0 \md{[d_1,l_1]} \\ n \cg -h \md{[d_2,l_2]}\\n \cg v_0\md{W}}} (2^{\Og(n)}+2^{\Og(n+h)}) \\
&=\sum\lm_{\st{(d_i,l_j)=1}} \lb_{d_1,d_2} \lb_{l_1,l_2}\tau[d_1,l_1]\sum\lm_{\st{n \sim \fr{x}{[d_1,l_1]}\\ n \cg \al' \md{[d_2,l_2]}\\ n \cg \bt' \md{W}}} 2^{\Og(n)} \quad + 
\quad \sum\lm_{\st{(d_i,l_j)=1}} \lb_{d_1,d_2} \lb_{l_1,l_2}\tau[d_2,l_2]
\sum\lm_{\st{n \sim \fr{x}{[d_2,l_2]} \\ n \cg \al'' \md{[d_1,l_1]} \\n \cg \bt'' \md{W}}} 2^{\Og(n)} \\
\end{split}
\end{equation}
Here, we have taken out $[d_1,l_1]$ and $[d_2,l_2]$ respectively from the 1st and 2nd term in the last equality of \ref{s2tent} and replaced the summation over $n+h$ with a summation over $n$. This is permitted because $h$ is very small compared to $x$. Note that therefore $\al'$ is coprime to $[d_2,l_2]$, $\al''$ is coprime to $[d_1,l_1]$ and $\bt',\bt''$ are both coprime to $W$. \\ \vskip 0.01in 
We assume that $\lb_{d_1,d_2}$'s are \textbf{symmetric} i.e $\lb_{d_1,d_2}=\lb_{d_2,d_1}$ for any $d_1,d_2$. This will allow us to interchange the indices $1$ and $2$ in the second term in the last equality of (\ref{s2tent}) to get
$$ S_2=\sum\lm_{\st{(d_i,l_j)=1}} \lb_{d_1,d_2} \lb_{l_1,l_2}\tau[d_1,l_1] 
\left( \sum\lm_{\st{n \sim \fr{x}{[d_1,l_1]}\\ n \cg \al' \md{[d_2,l_2]}\\ n \cg \bt' \md{W}}} 2^{\Og(n)} + \sum\lm_{\st{n \sim \fr{x}{[d_1,l_1]} \\ n \cg \al'' \md{[d_2,l_2]} \\n \cg \bt'' \md{W}}} 2^{\Og(n)}\right) $$

We make use of Corollary $\ref{2^Og}$ with $m=W[d_2,l_2]$ to evaluate the inner sums of the above expression. In order to apply the Corollary, the following conditions must hold
$$ W[d_2,l_2] \leq \left( \fr{x}{[d_1,l_1]} \right)^{2/3-\eps} \quad \text{and} \quad W[d_1,l_1] \leq \left( \fr{x}{[d_2,l_2]} \right)^{2/3-\eps} $$ 
Here, the second condition above holds because we have assumed $\lb_{d_1,d_2}$'s to be symmetric. So, their support is also symmetric in the indices $1$ and $2$. Since $W$ is very small as compared to $x$, it can be swallowed into the $x^{\eps}$ term. 
So it is enough to have
\begin{equation} \label{eqsupport} \max \{d_1^{2/3} d_2 , d_1d_2^{2/3} \} \leq z \end{equation}
where $ \ds z=x^{1/3-\eps}$. This explains why we have chosen the support of the sieve weights to be $S(z)$.\\ \vskip 0.01in
Now, Corollary \ref{2^Og} gives us
\begin{equation}
\label{s2}
\begin{split}
S_2&=2x\fr{\vp(W)}{W^2}c(W)\sum\lm_{\st{(d_i,l_j)=1}} \lb_{d_1,d_2} \lb_{l_1,l_2} 
\frac{\tau[d_1,l_1]}{[d_1,l_1]} \fr{1}{g[d_2,l_2]} \left( \log{x}-\log{[d_1,l_1]}+c'+ 
2\sum\lm_{\st{p \mid W[d_2,l_2] \\ p>2}} \frac{\log{p}}{p-2}\right) \\
&\quad+O_{\eps}\left(x^{1/2+\eps}\sum\lm_{\st{d_i,l_i}}\fr{|\lb_{d_1,d_2}||\lb_{l_1,l_2}|}{W^{1/4}[d_1,l_1]^{1/2}[d_2,l_2]^{1/4}} \right)
\end{split}
\end{equation}
where we have (with the notation of Corollary \ref{2^Og}), 
$$ g(p)=\fr{p(p-1)}{p-2} \quad \text{and} \quad c(W)=\prod\lm_{p \nmid W} \fr{(p-1)^2}{p(p-2)}$$
We can therefore write
\begin{equation}
\label{s2temp}
S_2=2x\fr{\vp(W)}{W^2}c(W) (M_2 - M_{21} + M_{22})+E_2
\end{equation}
where we have with $f(n)=n/\tau(n)$ that
\begin{equation}
\begin{split}
\label{identities}
M_2 &= (\log{x}+c')\sum\lm_{\st{(d_i,l_j)=1}}\fr{\lb_{d_1,d_2}\lb_{l_1,l_2}}{f[d_1,l_1]g[d_2,l_2]} \\
\quad M_{21} &= \sum\lm_{\st{(d_i,l_j)=1}} \fr{\lb_{d_1,d_2}\lb_{l_1,l_2}}{f[d_1,l_1]g[d_2,l_2]}\log{[d_1,l_1]}\\
M_{22}&=2\sum\lm_{\st{(d_i,l_j)=1}} \fr{\lb_{d_1,d_2}\lb_{l_1,l_2}}{f[d_1,l_1]g[d_2,l_2]}\sum\lm_{\st{p \mid W[d_2,l_2] \\ p>2}} \frac{\log{p}}{p-2} \\
E_2 &\ll_{\eps} \fr{x^{1/2+\eps}}{W^{1/4}} \sum\lm_{d_i,l_i}
\fr{|\lb_{d_1,d_2}| |\lb_{l_1,l_2}|}{[d_1,l_1]^{1/2}[d_2,l_2]^{1/4}}
\end{split}
\end{equation}
Out of all these terms, only the terms $M_2$ and $M_{21}$ will be contributing to the main term. The rest of the terms $M_{22}$ and $E_2$ will be error terms. \vskip 0.01in
To compute these terms, we prove a few Lemmas concerning additive functions. The purpose of doing this is to simplify the computation for $M_{21}$ and the estimation for $M_{22}$ with these Lemmas. In  both of these cases, Lemma \ref{useful} is directly applicable. In simpler words, Lemma \ref{useful} helps us give an asymptotic formula for $M_{21}$ and as well as we can use it to provide an upper bound for $M_{22}$.
\begin{lem}
\label{add}
Let $L(n)$ be an additive function defined on squarefree integers by $\ds L(n)=\sum\lm_{p \mid n}L(p)$. Let $C_{r_1,r_2}$ be as in $(\ref{BCr1r2})$. Then
$$\sum\lm_{\st{d_i \cg 0 \md{r_i}}}\fr{\lb_{d_1,d_2}}{f(d_1)g(d_2)}L(d_1)=\fr{C_{r_1,r_2}}{f_1(r_1)g_1(r_2)}L(r_1) + 
2\sum\lm_{\st{p}}\fr{C_{pr_1,r_2}}{f_1(r_1)g_1(r_2)}\fr{L(p)}{p-2}$$
where $f_1*1=f$ and $g_1*1=g$.
\end{lem}
\begin{proof}
We have
\begin{equation*}
\begin{split}
\sum\lm_{\st{d_i \cg 0 \md{r_i}}}\fr{\lb_{d_1,d_2}}{f(d_1)f(d_2)}L(d_1)&=\sum\lm_{\st{d_i \cg 0 \md{r_i}}}
\fr{\lb_{d_1,d_2}}{f(d_1)g(d_2)}\sum\lm_{p \mid d_1}L(p)=\sum\lm_{p}L(p)
\sum\lm_{\st{d_1 \cg 0\md{[p,r_1]}\\d_2 \cg 0 \md{r_2}}}\fr{\lb_{d_1,d_2}}{f(d_1)g(d_2)}\\
&=\sum\lm_{p}L(p) \fr{C_{[p,r_1],r_2}}{f_1[p,r_1]g_1(r_2)}=
\fr{C_{r_1,r_2}}{f_1(r_1)g_1(r_2)}L(r_1)+2\sum\lm_{p}\fr{C_{pr_1,r_2}}{f_1(r_1)g_1(r_2)}\fr{L(p)}{p-2}
\end{split}
\end{equation*}
In the last step above, we have split the sum into two cases, depending on whether $p \mid r_1$ or not. This completes the proof.
\end{proof}
\medskip
\begin{prop}
\label{useful}
Let $L(n)$ be as in Lemma $\ref{add}$ with the further restriction that $L(n) \ll \log{n}$. Then
\begin{equation*}
\begin{split}
\sum\lm_{\st{(d_i,l_j)=1}}\fr{\lb_{d_1,d_2}\lb_{l_1,l_2}}{f[d_1,l_1]g[d_2,l_2]}L[d_1,l_1]&=\sum\lm_{r_i}
\fr{C^2_{r_1,r_2}}{f_1(r_1)g_1(r_2)}L(r_1) +4\sum\lm_{r_i}\fr{C_{r_1,r_2}}{f_1(r_1)g_1(r_2)}\sum\lm_{p} \fr{L(p)}{p-2} C_{pr_1,r_2} + O \left( \fr{\log^6{z}}{D} \right)
\end{split}
\end{equation*}
\end{prop}
\begin{proof}
Now, 
\begin{equation*}
\begin{split}
LHS&=\sum\lm_{\st{(d_i,l_j)=1}}\fr{\lb_{d_1,d_2}\lb_{l_1,l_2}}{f[d_1,l_1]g[d_2,l_2]}L{[d_1,l_1]}=\sum\lm_{\st{(d_i,l_j)=1}}\fr{\lb_{d_1,d_2}\lb_{l_1,l_2}}{f[d_1,l_1]g[d_2,l_2]}(L(d_1)+L(l_1)-L(d_1,l_1))\\
&=\sd{}{_1}\sum \quad + \quad \sd{}{_2}\sum \quad-\quad \sd{}{_3}\sum
\end{split}
\end{equation*}
The fact that $L[d_1,l_1]=L(d_1)+L(l_1)-L(d_1,l_1)$ follows from the identity $[d_1,l_1]=\fr{d_1l_1}{(d_1,l_1)}$ and the fact that $L$ is additive.\\ \vskip 0.01in
\nd
As we have done while computing $S_1$, we want to get rid of the conditions $(d_1,l_2)=(d_2,l_1)=1$. So we multiply a factor
$$\sum\lm_{\st{s_{12} \mid (d_1,l_2)\\s_{21} \mid (d_2,l_1)}}\mu(s_{12})\mu(s_{21})$$ 
We then have the same conditions as before that $(r_1s_{12},r_2s_{21}), \ (r_1s_{21},r_2s_{12}) \in S(z)$.\\ \vskip 0.01in
Moreover, one can write 
$$\fr{1}{f[d_1,l_1]g[d_2,l_2]}=\fr{1}{f(d_1)f(l_1)g(d_2)g(l_2)} \sum\lm_{\st{r_i \mid (d_i,l_i)}}f_1(r_1)g_1(r_2)$$
Therefore,
\begin{equation}
\begin{split}
\sd{}{_1}\sum&=\sum\lm_{\st{d_i,l_i}}\fr{\lb_{d_1,d_2}\lb_{l_1,l_2}}{f(d_1)g(d_2)f(l_1)g(l_2)}L(d_1)
\sum\lm_{\st{r_i \mid (d_i,l_i)}}f_1(r_1)g_1(r_2)\sum\lm_{\st{s_{ij} \mid (d_i,l_j)}}\mu(s_{12})\mu(s_{21})\\
&=\sum\lm_{r_i,s_{ij}}\mu(s_{12})\mu(s_{21})f_1(r_1)g_1(r_2) \left( \sum\lm_{\st{l_i \cg 0 \md{r_i s_{ji}}}}\fr{\lb_{l_1,l_2}}{f(l_1)g(l_2)} \right)
\left(\sum\lm_{\st{d_i \cg 0\md{r_is_{ij}}}} \fr{\lb_{d_1,d_2}}{f(d_1)g(d_2)}L(d_1)\right)
\end{split}
\end{equation}
There are two bracketed quantities in the last term above. The first one is $C_{r_1s_{21},r_2s_{12}}$ divided by the quantity $f_1(r_1s_{21})g_1(r_2s_{12})$. For the second bracketed quantity, we invoke Lemma \ref{add}. We therefore obtain the expression 
\begin{equation}
\label{3.21}
\sd{}{_1}\sum=\sum\lm_{\st{r_i,s_{ij}}}\fr{\mu(s_{12}s_{21})}{f_1(s_{12}s_{21})g_1(s_{12}s_{21})}\fr{C_{r_1s_{21},r_2s_{12}}}{f_1(r_1)g_1(r_2)} \left( L(r_1s_{12})C_{r_1s_{12},r_2 s_{21}} + 2\sum\lm_{\st{p}}\fr{L(p)}{p-2}C_{pr_1s_{12},r_2s_{21}}\right)
\end{equation}
Above, we have used that fact that $f_1(p)=(p-2)/2$. Again, the contribution to $(\ref{3.21})$ from all those $(s_{12},s_{21})$ for which atleast one of $s_{ij} > D$ is 
$$\ll \sum\lm_{\st{(s_{12},s_{21}) \in S(z)\\s_{12}>D\\(r_1,r_2) \in S(z)}}\fr{1}{f_1(s_{12}s_{21})g_1(s_{12}s_{21})}\fr{\log^3{z}}{f_1(r_1)g_1(r_2)}  \ll \fr{\log^6{z}}{D}$$
where we have used the estimate $C_{r_1,r_2} \ll \log{z}$ from Proposition \ref{LBCr1r2}, the assumpion that $L(n) \ll \log{n}$ and the fact that $f_1 \in \Og_2$, $g_1 \in \Og_1$. \vskip 0.01in
So, we can assume $s_{12}=s_{21}=1$ at the cost of error $O\left( \fr{\log^6{z}}{D} \right)$. 
This gives us
\begin{equation*}
\sd{}{_1}\sum=\sum\lm_{r_i}\fr{C_{r_1,r_2}}{f_1(r_1)g_1(r_2)} 
\left( L(r_1)C_{r_1,r_2} + 2\sum\lm_{p \leq z}\fr{L(p)}{p-2}C_{pr_1,r_2} \right) + O\left( \fr{\log^6{z}}{D} \right)
\end{equation*}
By the same argument, one obtains
\begin{equation}
\sd{}{_2}\sum=\sd{}{_1}\sum
\end{equation}
We do similar operations as above for $\sd{}{_3}\sum$, We have
\begin{equation*}
\begin{split}
 \sd{}{_3}\sum &=\sum\lm_{\st{d_i,l_i}}\fr{\lb_{d_1,d_2}\lb_{l_1,l_2}}{f[d_1,l_1]g[d_2,l_2]}\sum\lm_{p \mid (d_1,l_1)}L(p)
\sum\lm_{\st{s_{ij} \mid (d_i,l_j)}}\mu(s_{12})\mu(s_{21})\\
& =\sum\lm_{\st{r_i,s_{ij}}}\mu(s_{12}s_{21})f_1(r_1)g_1(r_2) \sum\lm_{p} L(p) 
\left( \sum\lm_{\st{d_1 \cg 0\md{[p,r_1s_{12}]})\\d_2 \cg 0\md{r_2s_{21}}}} \fr{\lb_{d_1,d_2}}{f(d_1)g(d_2)} \right) 
\left( \sum\lm_{\st{l_1 \cg 0\md{[p,r_1s_{21}]})\\l_2 \cg 0\md{r_2s_{12}}}} \fr{\lb_{l_1,l_2}}{f(l_1)g(l_2)} \right)
\end{split}
\end{equation*}
Substituting the two bracketed quantities with appropriate $C_{r_1,r_2}$'s, we obtain
\begin{equation}
\sd{}{_3} \sum= \sum\lm_{\st{r_i,s_{ij}}} \fr{\mu(s_{12}s_{21})}{f_1(s_{12}s_{21})g_1(s_{12}s_{21})}\sum\lm_{p}
L(p)\fr{f_1(r_1)g_1(r_2)C_{[p,r_1s_{12}],r_2s_{21}}C_{[p,r_1s_{21}],r_2s_{12}}}{f_1[p,r_1s_{12}]f_1[p,r_1s_{21}] g_1(r_2s_{21})g_1(r_2s_{12})}
\label{absum}
\end{equation}

We now split the sum (\ref{absum}) into 4 cases depending on whether $p \mid r_1$, $p \mid s_{12}$, $p \mid s_{21}$ or $p \nmid r_1s_{12}s_{21}$. These are the only cases that can occur because we have $\mu^2(r_1s_{12}s_{21})=1$. We would then have
\begin{equation}
\begin{split}
\label{sg3}
\sd{}{_3} \sum &=\sum\lm_{\st{r_i,s_{ij}}}\fr{\mu(s_{12}s_{21})}{f_1(s_{12}s_{21})g_1(s_{12}s_{21})}
\fr{C_{r_1s_{12},r_2s_{21}}C_{r_1s_{21},r_2s_{12}}}{f_1(r_1)g_1(r_2)}L(r_1)\\
& \quad + 2\sum\lm_{\st{r_i,s_{ij}}}\fr{\mu(s_{12}s_{21})}{f_1(s_{12}s_{21})g_1(s_{12}s_{21})}\sum\lm_{\st{p}}\fr{L(p)}{(p-2)}\fr{C_{r_1s_{12},r_2s_{21}}C_{pr_1s_{21},r_2s_{12}}}{f_1(r_1)g_1(r_2)}\\
& \quad + 2\sum\lm_{\st{r_i,s_{ij}}}\fr{\mu(s_{12}s_{21})}{f_1(s_{12}s_{21})g_1(s_{12}s_{21})}
\sum\lm_{\st{p}}\fr{L(p)}{(p-2)}\fr{C_{pr_1s_{12},r_2s_{21}}C_{r_1s_{21},r_2s_{12}}}{f_1(r_1)g_1(r_2)}\\
& \quad + 4\sum\lm_{\st{r_i,s_{ij}}}\fr{\mu(s_{12}s_{21})}{f_1(s_{12}s_{21})g_1(s_{12}s_{21})}\sum\lm_{\st{p}}\fr{L(p)}{(p-2)^2}\fr{C_{pr_1s_{12},r_2s_{21}}C_{pr_1s_{21},r_2s_{12}}}{f_1(r_1)g_1(r_2)}
\end{split}
\end{equation}

Of the four terms in (\ref{sg3}) above, the 4th term is $O(\log^5{z})$. This follows by using the estimates $C_{r_1,r_2} \ll \log{z}$, $L(n) \ll \log{n}$ and the fact that $f_1 \in \Og_2$ and $g_1 \in \Og_1$. The 1st, 2nd and 3rd terms will be contributing to the main term. \vskip 0.01in
Again, the contribution to the 1st term, 2nd term and 3rd term of $(\ref{sg3})$ when atleast one of $s_{ij}>D$ is
$$ \ll \sum\lm_{\st{(s_{12},s_{21}) \in S(z)\\s_{12}>D}}\fr{1}{f_1(s_{12}s_{21})g_1(s_{12}s_{21})}
\sum\lm_{\st{(r_1,r_2) \in S(z)}}\fr{\log^3{z}}{f_1(r_1)g_1(r_2)} \ll 
\fr{\log^6{z}}{D}$$
So we may assume $s_{12}=s_{21}=1$ at the cost of error $O\left( \fr{\log^6{z}}{D} \right)$. Hence,
\begin{equation*}
\begin{split}
\sd{}{_3}\sum &= \sum\lm_{r_i}\fr{C^2_{r_1,r_2}}{f_1(r_1)g_1(r_2)}L(r_1) + 4 \sum\lm_{r_i}\fr{C_{r_1,r_2}}{f_1(r_1)g_1(r_2)}\sum\lm_{p}\fr{\log{p}}{p-2} C_{pr_1,r_2}+ O\left( \fr{\log^6{z}}{D} \right)
\end{split}
\end{equation*}
Putting together the expressions for $\sd{}{_1} \sum$, $\sd{}{_2} \sum$ and $\sd{}{_3} \sum$, we get the desired result.
\end{proof}
We move forward with with our computations. We compute $M_2$.
\medskip
\subsubsection*{Evaluation of \tps{$M_2$}{}}
We multiply a factor of $\sum\lm_{\st{s_{ij} \mid (d_i,l_j)}}\mu(s_{12})\mu(s_{21})$ to get rid of the conditions 
$(d_1,l_2)=(d_2,l_1)=1$. We then have as before that $\mu^2(r_1r_2s_{12}s_{21})=1$. Hence
\begin{equation}
\begin{split}
\fr{M_2}{\log{x}+c'}&=\sum\lm_{\st{(d_i,l_j)=1}} \fr{\lb_{d_1,d_2}\lb_{l_1,l_2}}{f[d_1,l_1]g[d_2,l_2]}=\sum\lm_{\st{d_i,l_i}}\fr{\lb_{d_1,d_2}}{f(d_1)g(d_2)}\fr{\lb_{l_1,l_2}}{f(l_1)g(l_2)} 
\sum\lm_{\st{r_i \mid (d_i,l_i)}}f_1(r_1)g_1(r_2) \sum\lm_{\st{s_{ij} \mid (d_i,l_j)}}\mu(s_{12})\mu(s_{21})\\
&=\sum\lm_{s_{ij}}\fr{\mu(s_{12}s_{21})}{f_1(s_{12}s_{21})g_1(s_{12}s_{21})}
\sum\lm_{r_i}\fr{C_{r_1s_{12},r_2}C_{r_1s_{21},r_2s_{12}}}{f_1(r_1)g_1(r_2)}
\end{split}
\end{equation}
As before, the contribution to the above sum when atleast one of $s_{ij}>D$ is $O \left( \fr{\log^5{z}}{D} \right)$.\\
So we may assume $s_{12}=s_{21}=1$ at the cost of error $O\left( \fr{\log^5{z}}{D} \right)$. Therefore,
\begin{equation*}
\begin{split}
\fr{M_2}{\log{x}+c'}&=\sum\lm_{r_i}\fr{C^2_{r_1,r_2}}{f_1(r_1)g_1(r_2)} +
O\left( \fr{\log^6{x}}{D} \right)\\
&=\sum\lm_{\st{(r_1,r_2) \in S(z)}}\fr{\mu^2(r_1)}{f_1(r_1)}\fr{\mu^2(r_2)}{g_1(r_2)}\left( B^2 h_1^2(r_1) h_2^2(r_2) Q_2^2\left( \fr{\log{r_1}}{\log{z}},\fr{\log{r_2}}{\log{z}} \right) + O(\log z) \right)+O\left( \fr{\log^6{x}}{D} \right)
\end{split}
\end{equation*}
where we have substituted the expression for $C_{r_1,r_2}$ from Lemma \ref{LBCr1r2}. By applying Corollary $\ref{sumS}$ to the inner sum of the main term above, we obtain
\begin{equation}
\begin{split}
\label{M2}
M_2&=B^5 \log x \iint\lm_{T \ \ } s_1Q_2^2(s_1,s_2)\,ds_2 \,ds_1 + O(\log^5 z)+O\left( \fr{\log^6 z}{D} \right)\\
&=(1+o(1))B^5 \log x \iint\lm_{T \ \ } s_1Q_2^2(s_1,s_2)\,ds_2 \,ds_1
\end{split}
\end{equation}
where $Q_2$ is as given in $(\ref{Q1Q2})$.
\medskip
\subsubsection*{Evaluation of \tps{$M_{21}$}{}}
Applying Proposition $\ref{useful}$ to the expression for $M_{21}$ in $(\ref{identities})$ with $L(n)=\log{n}$, we get
\begin{equation*}
\begin{split}
M_{21}&=\sum\lm_{r_i}\fr{C^2_{r_1,r_2}}{f_1(r_1)g_1(r_2)} \log{r_1} +4\sum\lm_{r_i}\fr{C_{r_1,r_2}}{f_1(r_1)g_1(r_2)}
\sum\lm_{p} \fr{\log{p}}{p-2} C_{pr_1,r_2} + O \left( \fr{\log^6{z}}{D} \right) \\
&=M_{21}^{(1)} \quad + \quad M_{21}^{(2)} \quad+ \quad  O \left( \fr{\log^6{z}}{D} \right)
\end{split}
\end{equation*}
Substituting the expression for $C_{r_1,r_2}$, we obtain
\begin{equation*}
M_{21}^{(1)}=\log{z}\sum\lm_{\st{(r_1,r_2) \in S(z)}}
\fr{\mu^2(r_1)}{f_1(r_1)}\fr{\mu^2(r_2)}{g_1(r_2)}\left( B^2 h_1^2(r_1) h_2^2(r_2) \fr{\log{r_1}}{\log{z}}Q_2^2\left( \fr{\log{r_1}}{\log{z}},\fr{\log{r_2}}{\log{z}} \right) + O(\log z) \right)
\end{equation*}
By Corollary $\ref{sumS}$ applied to the above sum, we get
\begin{equation}
\label{M21(1)}
M_{21}^{(1)}=(1+o(1))B^5\log{z}\iint\lm_{T \ \ } s_1^2Q_2^2(s_1,s_2)\,ds_2 \,ds_1
\end{equation}
Next, we look at $M_{21}^{(2)}$ and substitute the value of $C_{pr_1,r_2}$ and write $s_i=\fr{\log{r_i}}{\log{z}}$ to get
\begin{equation}
\begin{split}
\label{m21(2)}
M_{21}^{(2)}&=-4 \sum\lm_{\st{(r_1,r_2) \in S(z)}}\fr{\mu(r_1)\mu(r_2)C_{r_1,r_2}}{f_1(r_1)g_1(r_2)}
\sum\lm_{\st{(pr_1,r_2) \in S(z)}} \fr{(p-1)^2}{p(p-2)}\fr{\log{p}}{p}\left( B Q_2 \left( s_1+ \fr{\log{p}}{\log{z}},s_2 \right) + O(1) \right)\\
&= -4 \sum\lm_{\st{(r_1,r_2) \in S(z)}}\fr{\mu(r_1)\mu(r_2)C_{r_1,r_2}}{f_1(r_1)g_1(r_2)}\left[B \sum\lm_{\st{(pr_1,r_2) \in S(z)}} \fr{(p-1)^2}{p(p-2)}\fr{\log{p}}{p} Q_2 \left( s_1+ \fr{\log{p}}{\log{z}},s_2 \right) + O(\log z) \right]
\end{split}
\end{equation}
First, we focus on the inner sum of (\ref{m21(2)}). We have the condition $(pr_1,r_2) \in S(z)$ which is equivalent to
\begin{equation*}
p \leq z^{\eta\left(s_2\right)-s_1} \quad \text{and} \quad (p,Wr_1r_2)=1
\end{equation*}
with $\eta(s)$ as defined in (\ref{etadef}). Observing that $ \ds \fr{(p-1)^2}{p(p-2)}=1+O(1/p^2)$, it follows that the inner sum of (\ref{m21(2)}) is
\begin{equation*}
B \sum\lm_{\st{2 <p \leq z^{\eta(s_2)-s_1}}} \fr{\log{p}}{p} Q_2 \left( s_1+ \fr{\log{p}}{\log{z}},s_2 \right)
+O\left( B \sum\lm_{p \mid Wr_1r_2}\fr{\log{p}}{p} \right)+O(\log z)
\end{equation*}
By Proposition \ref{logp/logz}, the main term above becomes
\begin{equation*}
\begin{split}
&\quad B (\log z+O(1))\int\lm_{0}^{\eta(s_2)-s_1}Q_2(s_1+t,s_2)\,dt =B(\log z+O(1))\int\lm_{s_1}^{\eta(s_2)} Q_2(t_1,s_2)\\
&= B(\log z+O(1)) \int\lm_{s_1}^{\eta(s_2)} \int\lm_{s_2}^{\eta(s_1)} P(t_1,t_2) \, dt_2 \, dt_1=B(\log z+O(1)) \ Q_1(s_1,s_2)
\end{split}
\end{equation*}
Plugging the expression for the inner sum back into (\ref{m21(2)}), we have
\begin{equation}
\begin{split}
\label{m212temp}
M_{21}^{(2)}&=-4\sum\lm_{\st{(r_1,r_2) \in S(z)}}\fr{\mu^2(r_1)\mu^2(r_2)}{f_1(r_1) g_1(r_2)}\\
&\quad \times \left( B^2 \log z \ h_1(r_1)h_2(r_2)\ Q_1\left( \fr{\log r_1}{\log z},\fr{\log r_2}{\log z} \right) Q_2\left( \fr{\log r_1}{\log z},\fr{\log r_2}{\log z} \right)  + O(\log^2 z) \right)\\
 & \quad + O\left(\log^2{z}\sum\lm_{\st{(r_1,r_2) \in S(z)}} \fr{h_1(r_1)}{f_1(r_1)}\fr{h_2(r_2)}{g_1(r_2)}\sum\lm_{p \mid Wr_1r_2}\fr{\log{p}}{p} \right)
\end{split}
\end{equation}
The second error term in $(\ref{m212temp})$ turns out to be $O(\log^5{z})$. This is done by replacing 
$$\sum\lm_{(r_1,r_2) \in S(z)} \ \text{by} \ \ll \sum\lm_{\st{r_1 \leq z\\r_2 \leq z}} \quad \text{and}  \quad \sum\lm_{p \mid Wr_1r_2} \ \text{by} \ \sum\lm_{p \mid W}+\sum\lm_{p \mid r_1}+\sum\lm_{p \mid r_2}$$ 
and interchanging the order of summation. Coming back to the main term in (\ref{m212temp}), apply Corollary $\ref{sumS}$ to the inner sum and get
\begin{equation}
\begin{split}
\label{M21(2)}
M_{21}^{(2)}&=-4(1+o(1))B^5 \log{z}\iint\lm_{T \ \ } s_1Q_1(s_1,s_2)Q_2(s_1,s_2) \,ds_2 \,ds_1 
\end{split}
\end{equation}
\subsubsection*{Estimation of \tps{$M_{22}$}{}}
For the estimation of $M_{22}$, we can apply Proposition \ref{useful} with $\ds L(n)=\sum\lm_{\st{p \mid n\\p>2}}\fr{\log{p}}{p-2}$ and obtain
\begin{equation}
\label{m22temp}
M_{22}=\sum\lm_{r_i} \fr{C_{r_1,r_2}}{f_1(r_1)g_1(r_2)} \left( C_{r_1,r_2} \left(\sum\lm_{\st{p \mid r_1\\p>2}}\fr{\log{p}}{p-2}\right) + 4\sum\lm_{\st{2<p \leq z}}\fr{\log{p}}{(p-2)^2}C_{pr_1,r_2} \right) + O\left( \fr{\log^6{z}}{D} \right)
\end{equation}
Using the estimate $C_{r_1,r_2} \ll \log{z}$ in (\ref{m22temp}) we get 
\begin{equation}
\begin{split}
\label{M22}
M_{22} & \ll \sum\lm_{\st{r_1 \leq z\\r_2 \leq z}}\fr{ \log^2{z}}{f_1(r_1)g_1(r_2)} 
\left(\sum\lm_{\st{p \mid r_1\\p>2}}\fr{\log{p}}{p-2} + \sum\lm_{2<p \leq z}\fr{\log{p}}{(p-2)^2} \right) + O\left( \fr{\log^{z}}{D} \right)\\
& \ll \log^2{z}\left( \sum\lm_{r_2 \leq z}\fr{1}{g_1(r_2)} \right) \left( \sum\lm_{2<p \leq z}\fr{\log{p}}{p(p-2)}
\sum\lm_{\st{r_1 \leq \fr{z}{p}}}\fr{1}{f_1(r_1)} \right)+\log^5{z}+\fr{\log^6{z}}{D}\\
& \ll \fr{\log^6{z}}{D}
\end{split}
\end{equation}
\subsubsection*{Estimation of \tps{$E_2$}{}}
From (\ref{identities}), we have
$$ E_2 \ll_{\eps} \fr{x^{1/2+\eps}}{W^{1/4}} \sum\lm_{\st{d_i,l_i}}
\fr{|\lb_{d_1,d_2}| |\lb_{l_1,l_2}|}{[d_1,l_1]^{1/2}[d_2,l_2]^{1/4}} $$
Using the estimate $\lb_{d_1,d_2} \ll \log^4{z}$ from Lemma \ref{LBCr1r2} and writing $\ds [d_i,l_i]=\fr{d_il_i}{(d_i,l_i)}$, we get
$$ E_2 \ll_{\eps} \fr{x^{1/2+\eps} \log^8{z}}{W^{1/4}}\sum\lm_{\st{(d_1,d_2) \in S(z)\\(l_1,l_2) \in S(z)}}
\fr{(d_1,l_1)^{1/2}(d_2,l_2)^{1/4}}{d_1^{1/2}d_2^{1/4}l_1^{1/2} l_2^{1/4}} $$
Since $W^{1/4} \ll (\log\log{x})^{1/4}$ and $\log^8{z}$ are small, these terms can be swallowed into the term $x^{\eps}$. Therefore,
\begin{equation}
\begin{split}
\label{e2tmp}
E_2 &\ll_{\eps} x^{1/2+\eps} \sum\lm_{\st{(g_1,g_2) \in S(z)}} g_1^{1/2}g_2^{1/4} \sum\lm_{\st{(d_1,d_2) \in S(z)\\(l_1,l_2) \in S(z)\\(d_i,l_i)=g_i}}d_1^{-1/2}d_2^{-1/4}l_1^{-1/2}l_2^{-1/4}\\
& \ll_{\eps} x^{1/2+\eps} \sum\lm_{(g_1,g_2) \in S(z)} g_1^{-1/2} g_2^{-1/4} 
\left(\sum\lm_{\st{d_1 \leq z/g_1 \\ g_2d_2 \leq z/(g_1d_1)^{2/3}}}d_1^{-1/2}d_2^{-1/4} \right)^2 \\
\end{split}
\end{equation}
Now,
\begin{equation*}
\begin{split}
\sum\lm_{\st{d_1 \leq z/g_1 \\ g_2d_2 \leq z/(g_1d_1)^{2/3}}} d_1^{-1/2} d_2^{-1/4}&=\sum\lm_{d_1 \leq z/g_1}d_1^{-1/2}\sum\lm_{g_2d_2 \leq z/(g_1d_1)^{2/3}}d_2^{-1/4} \ll \sum\lm_{d_1 \leq z/g_1}d_1^{-1/2}\left( \fr{z}{g_2(g_1d_1)^{2/3}} \right)^{3/4}\\
&\ll z^{3/4}g_1^{-1/2}g_2^{-3/4}\sum\lm_{d_1 \leq z/g_1}d_1^{-1} \ll g_1^{-1/2}g_2^{-3/4} z^{3/4} \log{z} 
\end{split}
\end{equation*}
Substituting the above expression into (\ref{e2tmp}), we obtain
\begin{equation*}
E_2 \ll_{\eps} x^{1/2+\eps} z^{3/2} \log^2{z} \sum\lm_{\st{g_1 \leq z\\g_2 \leq z/g_1^{2/3}}}g_1^{-3/2}g_2^{-7/4} 
\ll_{\eps} x^{1/2+\eps} z^{3/2} \log^2{z}
\end{equation*}
Under the choice $z=x^{1/3-\eps}$, we get 
\begin{equation} \label{E_2} E_2 \ll_{\eps} x^{1-\eps/2} \end{equation}

Substituting $z=x^{1/3-\eps}$ and noting that $c(W)=1+o(1)$, we get from (\ref{M2}), (\ref{M21(1)}), (\ref{M21(2)}), (\ref{M22}), (\ref{E_2}) and (\ref{s2temp}) that
\begin{equation}
\label{S2}
S_2=\fr{2x}{W}(1+o(1)) B^6 R_2(P)
\end{equation}
where
\begin{equation}
\label{R2}
R_2(P)=\iint\lm_{T \ \ }s_1(3+\eps-s_1)Q_2^2(s_1,s_2) \,ds_2 \,ds_1 + 4\iint\lm_{T \ \ }s_1Q_1(s_1,s_2)Q_2(s_1,s_2) \, ds_2 \, ds_1
\end{equation}
\begin{rem}
In the above expression (\ref{R2}) for $R_2(P)$, there is an $\eps$ occuring. Since $\eps$ can be made arbitrarily small, we do not consider it for our computations.
\end{rem}

This completes the Proof of Proposition \ref{asymptote}.
\bigskip

\subsection{The Value of \tps{$\lb$}{}}
Proposition \ref{asymptote} gives us the asymptotic expressions for $S_1$ and $S_2$ in terms of the symmetric differentiable (in each variable) function $P$. There are two more functions, namely $Q_1$ and $Q_2$ defined in terms of P (See \ref{Q1Q2}). We make use of a computer program to optimise the choice of $P$ to determine a suitable value of $\lb$. \vskip 0.01in
Let $\eta(s)$ be as defined in (\ref{etadef}). Recall that
\begin{equation}
\label{obs} \eta(s)=\begin{cases} 1-2s/3 & \text{if} \ s \leq 3/5 \\ 3/2(1-s) & \text{if} \ s \geq 3/5 \end{cases} \end{equation}
Moreover, note that for any integrable function $F:T \to \mb{R}$, one can write
\begin{equation}
\label{intT} 
\iint\lm_{T_{s_1,s_2} \ \ } F(t_1,t_2) \, dt_2 \, dt_1 =\int\lm_{s_1}^{\eta(s_2)} \int\lm_{s_2}^{\eta(t_1)} F(t_1,t_2) \, dt_2 \, dt_1
\end{equation}
Using the above observations in (\ref{obs}) and (\ref{intT}) and using the expressions for $Q_1$ and $Q_2$ are in terms of $P$, from (\ref{Q1Q2}), we obtain
\begin{align}
Q_1(s_1,s_2)&=
\begin{dcases}
\int\lm_{s_1}^{\fr{3}{5}} \int\lm_{s_2}^{1-\fr{2}{3}t_1} P(t_1,t_2) \, dt_2 \, dt_1 + \int\lm_{\fr{3}{5}}^{1-\fr{2}{3}s_2} 
\int\lm_{s_2}^{\fr{3}{2}(1-t_1)}P(t_1,t_2) \, dt_2 \, dt_1 & \ 0 \leq s_1, \ s_2 \leq \fr{3}{5} \\
\int\lm_{s_1}^{\fr{3}{2}(1-s_1)} \int\lm_{s_2}^{1-\fr{2}{3}t_1} P(t_1,t_2) \, dt \,_2 dt_1 & \ 0 \leq s_1 \leq \fr{3}{5} \leq s_2 \leq 1\\
\int\lm_{s_1}^{1-\fr{2}{3}s_2} \int\lm_{s_2}^{\fr{3}{2}(1-t_1)} P(t_1,t_2) \, dt_2 \, dt_1 & \ 0 \leq s_2 \leq \fr{3}{5} \leq s_1 \leq 1
\end{dcases} \label{Q1int}\\
Q_2(s_1,s_2)&=
\begin{dcases}
\int\lm_{s_2}^{1-\fr{2}{3}s_1} P(s_1,t_2) \, dt_2 & \ 0 \leq s_1 \leq \fr{3}{5} \\
\int\lm_{s_2}^{\fr{3}{2}(1-s_1)}P(s_1,t_2) \, dt_2 & \ \fr{3}{5} \leq s_1 \leq 1
\end{dcases} \label{Q2int}
\end{align}
Applying the same observations (\ref{obs}), (\ref{intT}) to the expression for $R_1(P)$ and $R_2(P)$ from Proposition \ref{asymptote}, we obtain the following (rather complicated) expressions
\begin{equation}
\label{R1(P)int}
\begin{split}
R_1(P)&=\int\lm_{0}^{\fr{3}{5}} \int\lm_{0}^{\fr{3}{5}} \left( \int\lm_{s_1}^{\fr{3}{5}} \int\lm_{s_2}^{1-\fr{2}{3}t_1} P(t_1,t_2)\,  dt_2 \, dt_1 + \int\lm_{\fr{3}{5}}^{1-\fr{2}{3}s_2} \ \int\lm_{s_2}^{\fr{3}{2}(1-t_1)}P(t_1,t_2) \, dt_2 \, dt_1 \right)^2 \, ds_2 \, ds_1\\
& \quad + \int\lm_{0}^{\fr{3}{5}} \int\lm_{\fr{3}{5}}^{1-\fr{2}{3}s_1} \left( \int\lm_{s_1}^{\fr{3}{2}(1-s_1)} \int\lm_{s_2}^{1-\fr{2}{3}t_1} P(t_1,t_2) \, dt_2 dt_1 \right)^2 \, ds_2 \, ds_1\\
& \quad +\int\lm_{\fr{3}{5}}^1 \int\lm_{0}^{\fr{3}{2}(1-s_1)} \left( \int\lm_{s_1}^{1-\fr{2}{3}s_2} \int\lm_{s_2}^{\fr{3}{2}(1-t_1)} P(t_1,t_2) \, dt_2 \, dt_1 \right)^2 \, ds_2 \, ds_1
\end{split}
\end{equation}
\begin{equation}
\label{R2(P)int}
\begin{split}
R_2(P)&=\int\lm_{0}^{\fr{3}{5}}\int\lm_{0}^{1-\fr{2}{3}s_1} s_1(3-s_1) \left( \int\lm_{s_2}^{1-\fr{2}{3}s_1} P(s_1,t_2) \, dt_2 \right)^2 \, ds_2 \, ds_1  + \int\lm_{\fr{3}{5}}^1 \int\lm_{0}^{\fr{3}{2}(1-s_1)} s_1(3-s_1) \left( \int\lm_{s_2}^{\fr{3}{2}(1-s_1)}P(s_1,t_2) \, dt_2 \right)^2 \, ds_2 \, ds_1\\
& \quad + 4\int\lm_{0}^{\fr{3}{5}} \int\lm_{0}^{\fr{3}{5}} s_1 \left( \int\lm_{s_2}^{1-\fr{2}{3}s_1} P(s_1,t_2) \, dt_2 \right)
\left( \int\lm_{s_1}^{\fr{3}{5}} \int\lm_{s_2}^{1-\fr{2}{3}t_1} P(t_1,t_2) \, dt_2 \, dt_1 + \int\lm_{\fr{3}{5}}^{1-\fr{2}{3}s_2} \ \int\lm_{s_2}^{\fr{3}{2}(1-t_1)}P(t_1,t_2) \, dt_2 \, dt_1 \right) \, ds_2 \, ds_1 \\
&\quad + 4\int\lm_{0}^{\fr{3}{5}} \int\lm_{\fr{3}{5}}^1 s_1 \left( \int\lm_{s_2}^{1-\fr{2}{3}s_1} P(s_1,t_2) \, dt_2 \right) 
\left( \int\lm_{s_1}^{\fr{3}{2}(1-s_1)} \int\lm_{s_2}^{1-\fr{2}{3}t_1} P(t_1,t_2) \, dt \,_2 dt_1 \right) \, ds_2 \, ds_1 \\
& \quad + 4 \int\lm_{\fr{3}{5}}^1 \int\lm_{0}^{\fr{3}{5}} s_1 \left( \int\lm_{s_2}^{\fr{3}{2}(1-s_1)}P(s_1,t_2) \, dt_2 \right) \left( \int\lm_{s_1}^{1-\fr{2}{3}s_2} \int\lm_{s_2}^{\fr{3}{2}(1-t_1)} P(t_1,t_2) \, dt_2 \, dt_1 \right) \, ds_2 \, ds_1
\end{split}
\end{equation}
Keeping in mind that $P$ should be symmetric, we make the following choice for $P$. We set
\begin{equation} \label{Pchoice} P(x,y)=\sum\lm_{i,j=0}^7 a_{i,j} (x+y)^i (x^2+y^2)^j \end{equation}
The choice of the coefficients will optimised by means of a sage program. We evaluate $R_1(P)$ and $R_2(P)$ in terms of these coefficients. We first convert the coefficients $a_{i,j}$ into variables of the form $\{a_0,a_1,a_2, \dots , a_{63}\}$ via the bijective map $(i,j)\mapsto 8i+j$ and then integrate to obtain the expressions $R_1$ and $R_2$.
\begin{verbatim}
n=7
N=(n+1)^2
x=var("x")
y=var("y")
def pos(n,i,j):
    return (n+1)*i+j
aa=list(var('a_%d ' % i) for i in (0..(N)) )
p=0
for i in range(0,n+1):
    for j in range(0,n+1):
        p=p+aa[pos(n,i,j)]*(x+y)^i*(x^2+y^2)^j 
								
t1=var('t1')
t2=var('t2')
s1=var('s1')
s2=var('s2')
t=var('t')

P(a,b)=p(x=a,y=b)

Q11(s1,s2)=integrate(integrate(P(t1,t2),t2,s2,1-2/3*t1),t1,s1,3/5)
						+integrate(integrate(P(t1,t2),t2,s2,3/2-3/2*t1),t1,3/5,1-2/3*s2)
						
Q12(s1,s2)=integrate(integrate(P(t1,t2),t2,s2,1-2/3*t1),t1,s1,3/2-3/2*s2)

Q13(s1,s2)=integrate(integrate(P(t1,t2),t2,s2,3/2-3/2*t1),t1,s1,1-2/3*s2)
        
R1=integrate(integrate(Q11(s1,s2)^2,s2,0,3/5),s1,0,3/5)
+integrate(integrate(Q12(s1,s2)^2,s2,3/5,1-2/3*s1),s1,0,3/5)
+integrate(integrate(Q13(s1,s2)^2,s2,0,3/2-3/2*s1),s1,3/5,1)

Q21(s1,s2)=integrate(P(s1,t2),t2,s2,1-2/3*s1)  
  
Q22(s1,s2)=integrate(P(s1,t2),t2,s2,3/2-3/2*s1)

Q41(s1,s2)=s1*(3-s1)*Q21(s1,s2)^2+4*s1*Q11(s1,s2)*Q21(s1,s2)

Q42(s1,s2)=s1*(3-s1)*Q21(s1,s2)^2+4*s1*Q12(s1,s2)*Q21(s1,s2)

Q43(s1,s2)=s1*(3-s1)*Q22(s1,s2)^2+4*s1*Q13(s1,s2)*Q22(s1,s2)
    
R2=integrate(integrate(Q41(s1,s2),s2,0,3/5),s1,0,3/5)
+integrate(integrate(Q42(s1,s2),s2,3/5,1-2/3*s1),s1,0,3/5)
+integrate(integrate(Q43(s1,s2),s2,0,3/2-3/2*s1),s1,3/5,1)

R2=R2.expand()
R1=R1.expand()
\end{verbatim}
In this code, $R_1$ and $R_2$ have expressions for $R_1(P)$ and $R_2(P)$ stored in them. These expressions are actually quadratic forms in the coefficients $a_0,a_1, \dots, a_{63}$.\\ \vskip 0.02in

To show that $(\ref{thesum})$ is positive, we must have $\lb S_1>S_2$. By Proposition \ref{asymptote}, we must have

\begin{equation}
\label{L}
\lb>\fr{2  R_2(P)}{ R_1(P)}
\end{equation}
So, we need to minimize $R_2(P)/R_1(P)$. This ratio is entirely dependent on the choice of function $P$ and since it so happens with this choice of $P$ (See \ref{Pchoice}) that both $R_2$ and $R_1$ are quadratic forms in the coefficients 
$a_0,a_1, \dots, a_{63}$, we are essentially looking to minimize the ratio of two quadratic forms. This is a well known problem with the following solution.
\begin{thm}
\label{minratio}
Let $R_1=\mathbf{a}^T M_1\mathbf{a}$ and $R_2=\mathbf{a}^T M_2\mathbf{a}$ be two quadratic forms, where $M_1$ and $M_2$ are positive definite real symmetric matrices. Then the ratio $R_2/R_1$ is minimized when $\mathbf{a}$ is an eigenvector corresponding to the smallest eigenvalue of $M_1^{-1}M_2$. The value of the ratio at its minimum is this minimum eigenvalue. 
\end{thm}
\begin{proof}
For a proof, see Lemma 7.3 \cite[Pg 20]{JM}.
\end{proof}
We minimize this ratio using a sage program. We will let $A$ to be the matrix corresponding to the form $R_1$ and $B$ corresponding to $R_2$. We define them as follows
$$ A_{i,j}=\fr{\partial^2 R_1}{\partial a_i \partial a_j} \quad \text{and} \quad B_{i,j}=\fr{\partial^2 R_2}{ \partial a_i \partial a_j} $$
Then it is clear that $\ds \mathbf{a}^T A \mathbf{a}=2R_1$ and $\ds \mathbf{a}^T B \mathbf{a}=2R_2$. Moreover, the matrices $A$ and $B$ are clearly real symmetric matrices. If one looks at the expressions (\ref{s11}) and (\ref{s22}) for $S_1$ and $S_2$ respectively, it clear why both the matrices $A$ and $B$ (for $R_1(P)$ and $R_2(P)$ respectively) are positive definite. \vskip 0.01in
So by Theorem \ref{minratio}, the the minimum value of $R_2(P)/R_1(P)$ is the smallest eigenvalue of $C$, which is computed below.

\begin{verbatim}
A=matrix(QQ,N)
B=matrix(QQ,N)
for i in range(0,N):
    for j in range(0,N):
        A[i,j]=derivative(derivative(R1,aa[i]),aa[j])
        B[i,j]=derivative(derivative(R2,aa[i]),aa[j])
C=A.inverse()*B
min(C.eigenvalues())
\end{verbatim}
\[
6.290731135292344?
\]
Therefore, from (\ref{L}), it follows that (\ref{thesum}) is positive for any
\begin{equation}
\label{Lb}
\lb>12.5814622705847
\end{equation}
\vskip 0.3in
\begin{rem}
It is of course possible to reduce the value of $\lb$ somewhat further, if we take $P(x,y)$ to be of a higher degree. But the improvement obtained will be very little and $\lb$ is not likely to go below $12$. Plus, the process can become increasingly time consuming for any computer program as we increase the degree of the polynomial $P(x,y)$.
\end{rem}
\vskip 0.2in
\subsection*{Acknowledgements}
The authors are thankful to the Institute of Mathematical Sciences, Chennai, where all the work has been carried out, for its hospitality and for providing an ideal research atmosphere.
\vskip 0.3in

\end{document}